%% file: main.tex
\documentclass[a4paper,10pt]{article}
\usepackage[a4paper, inner=2cm, outer=2cm, top=2.3cm, bottom=2.3cm]{geometry}

\usepackage{amssymb,amsthm,amsmath}
\usepackage{mathrsfs}
\usepackage{mathtools}
\usepackage[bookmarks]{hyperref}
\usepackage{latexsym}
\usepackage{enumerate}
\usepackage[all]{xy}
\usepackage[T1]{fontenc}
\usepackage{indentfirst}
\usepackage{color}
\usepackage{graphicx}
\usepackage{verbatim}
\usepackage{comment}
\usepackage{bm}
\usepackage{tikz}
\usepackage{tikz-cd}
\usetikzlibrary{decorations.pathmorphing}
\usepackage{changepage} 
\usepackage{lipsum} 
\usepackage{url}

\usepackage{changepage} 
\usepackage{indentfirst}
\usepackage{float}
\usepackage{physics}

\newcommand{\hyperplanesBtwo}{
    \begin{scope}[black!20]
        \clip (-2.3,-2.3) -- (-2.3,2.3) -- (2.3,2.3) -- (2.3,-2.3) -- (-2.3,-2.3);
        \foreach \i in {-3,...,3}{
        \draw
            (-3,\i)--(3,\i)
            (\i,-3)--(\i,3)
            (-4,\i*2)--(4,\i*2+8)
            (\i*2,-4)--(\i*2+8,4)
            (-5,\i*2)--(3,\i*2+8)
            (\i*2,-5)--(\i*2+8,3)
            (4,\i*2)--(-4,\i*2+8)
            (\i*2,-4)--(\i*2-8,4)
            (5,\i*2)--(-3,\i*2+8)
            (\i*2,-5)--(\i*2-8,3);
        }
    \end{scope}
}

\newcommand{\an}{\textnormal{an}}

\newcommand{\der}{\textnormal{der}}

\newcommand{\Op}{\textnormal{op}}
\newcommand{\nr}{\textnormal{nr}}

\newcommand{\rs}{\textnormal{rs}}
\newcommand{\Sc}{\textnormal{sc}}

\newcommand{\un}{\textnormal{un}}

\newcommand{\vol}{\textnormal{vol}}
\newcommand{\red}{\textnormal{red}}

\DeclareMathOperator{\Mod}{Mod}
\DeclareMathOperator{\Rep}{Rep}
\DeclareMathOperator{\Irr}{Irr}

\DeclareMathOperator{\val}{val}

\DeclareMathOperator{\Ad}{Ad}

\DeclareMathOperator{\cusp}{cusp}

\DeclareMathOperator{\End}{End}
\DeclareMathOperator{\Fr}{Fr}
\DeclareMathOperator{\Frob}{Frob}
\DeclareMathOperator{\Gal}{Gal}

\DeclareMathOperator{\Hom}{Hom}
\DeclareMathOperator{\Ind}{Ind}

\DeclareMathOperator{\SL}{SL}
\DeclareMathOperator{\SU}{SU}

\DeclareMathOperator{\im}{Im}

\DeclareMathOperator{\aff}{aff}
\DeclareMathOperator{\ind}{Ind}

\DeclareMathOperator{\fdeg}{fdeg}

\DeclareMathOperator{\Lie}{Lie}

\DeclareMathOperator{\cond}{cond}
\DeclareMathOperator{\triv}{triv}
\DeclareMathOperator{\Stab}{Stab}
\DeclareMathOperator{\ord}{ord}

\newtheorem{theorem}{Theorem}[section]
\newtheorem{lemma}[theorem]{Lemma}
\newtheorem{proposition}[theorem]{Proposition}

\theoremstyle{definition}
\newtheorem{definition}[theorem]{Definition}
\theoremstyle{remark}
\newtheorem{remark}[theorem]{Remark}
\newtheorem{example}[theorem]{Example}

\newtheorem{conjecture}{Conjecture}[section]

\title{Formal degree of principal series of quasi-split groups}
\author{Giulio Ricci\\
Radboud Universiteit Nijmegen\\
Heyendaalseweg 135, 6525AJ Nijmegen, the Netherlands\\
\texttt{giulio.ricci@ru.nl}}
\date{\today}

\begin{document}

\maketitle

\begin{abstract}   
    Let $\mathcal{G}$ be a quasi-split connected reductive group over a non-archimedean local field $F.$ In this paper, we prove the formal degree conjecture for discrete series representations contained in a principal series of $\mathcal{G}(F)$. We first construct a type for each Bernstein component attached to a principal series representation of $\mathcal{G}(F).$ We then use these types and the local Langlands correspondence for principal series representations constructed by Solleveld to verify the formal degree conjecture. Our approach follows a similar strategy to the case of split groups, reducing the problem to the case of unipotent representations of some endoscopic group.
\end{abstract}

\tableofcontents

\section{Introduction}
    Let $\mathcal{G}$ be a connected reductive quasi-split group over a non-archimedean local field $F$ with residue field of order $q_F$ and characteristic $p,$ and let $\mathcal{T}\subset \mathcal{G}$ be a maximal torus in $\mathcal{G}$ containing a maximal split torus $\mathcal{S}$. Write $G=\mathcal{G}(F)$, $T=\mathcal{T}(F)$ and $S=\mathcal{S}(F).$ The aim of this paper is to prove the formal degree conjecture for discrete series representations contained in a principal series of $G$. This generalizes the previous work of the author \cite{Ricci} where the same result was proved for split groups. 
    
    The formal degree conjecture, formulated in 2007 by Hiraga, Ichino, and Ikeda \cite[Conjecture 1.4]{HII}, predicts a relationship between the formal degree of a discrete series representation and the adjoint $\gamma$-factor associated with its $L$-parameter whenever a local Langlands correspondence is known. 

    We write $G^{\vee}$ for the complex dual group of $G$. Let $Z(G)_s$ be the the $F$-points of the maximal $F$-split central torus in $G.$ For $\varphi$ an $L$-parameter of $G^{\vee,}$ we write $$S_{\varphi}^{\sharp}:=\pi_0(Z_{(G/Z(G)_s)^{\vee}}(\varphi))$$ for the component group of the centralizer of $\varphi$ in $(G/Z(G)_s)^{\vee}.$ We fix an order-0 additive character $\psi$ of $F$ and attach to $\psi$ and $\varphi$ the adjoint $\gamma$-factor $\gamma(0,\Ad_{G^{\vee}}\circ\varphi,\psi),$ as defined in \cite{ArithmeticInv}.

    \begin{conjecture}\cite[1.4]{HII}\label{HII}
        Assume we have a local Langlands correspondence for $G$, and let $\pi$ be an irreducible discrete series representation of $G$. Let $(\varphi_{\pi},\rho_{\pi})$ be the enhanced $L$-parameter associated to it and let $\psi$ be an order-0 additive character of $F$. Then 
        $$\fdeg(\pi)=\frac{\dim(\rho_{\pi})}{\abs{S_{\varphi_\pi}^{\sharp}}}\abs{\gamma(0,\Ad_{G^{\vee}}\circ\varphi_{\pi},\psi)}.$$
    \end{conjecture}

\begin{theorem}\label{FirstTheorem}
     Let $\mathcal{G}$ be a connected reductive quasi-split group over $F.$ Consider the local Langlands correspondence for principal series representations of quasi-split groups as in \cite{quasisplitps}. Then Conjecture \ref{HII} is true for every irreducible discrete series contained in a parabolically induced representation from a maximal torus of $\mathcal{G}(F)$ containing a maximal split torus.
\end{theorem}

    In \cite{quasisplitps}, a local Langlands correspondence for principal series representations of quasi-split groups is defined. The correspondence goes through an equivalence of affine Hecke algebras. After fixing a Bernstein block $\mathfrak{s}=[T,\chi]_G,$ the algebra of endomorphisms of a progenerator of the Bernstein block $\Rep(G)^{\mathfrak{s}}$ is matched with the affine Hecke algebra $\mathcal{H}(\mathfrak{s}^{\vee},q_F^{1/2})$ attached to a Bernstein component $\mathfrak{s}^{\vee}$ of $L$-parameters. The Bernstein decomposition of $L$-parameters and the affine Hecke algebras associated with one of these Bernstein components were defined by Aubert, Moussaoui, and Solleveld \cite{AffineHeckeForLparameters}. In the same paper, modules over $\mathcal{H}(\mathfrak{s}^{\vee},q_F^{1/2})$ are then matched with $L$-parameters. 
    
    Checking how the formal degree changes when moving from $\Rep(G)$ to the category of modules over $\mathcal{H}(\mathfrak{s}^{\vee},q_F^{1/2})$ is harder for quasi-split groups than it is for split groups. The main problem is the lack of a type for the Bernstein component corresponding to any principal series of a quasi-split group. The construction of types for principal series of split groups done by Roche \cite{RochePrincipal} does not generalize to quasi-split groups straightforwardly. Types for quasi-split groups were known only in specific cases, like $U_n$ \cite{MyauchiStevens} or depth-0 representations \cite{morris1999level}. In Section \ref{TypeSection}, we construct an explicit type for the Bernstein component associated with any principal series representation of a quasi-split $p$-adic group, under a mild assumption on $p$.

   After constructing suitable types for principal series representations, the computation of the formal degree proceeds along lines very similar to the split case \cite[Theorem 4.7]{Ricci}. To a discrete series representation $\pi\in\Rep(G)^{\mathfrak{s}},$ we associate a unipotent representation $\pi_J$ of a quasi-split group $\mathcal{J}(F)$ satisfying
   $$\mathcal{J}^{\vee}=Z_{G^{\vee}}(\varphi_\chi(I_F))^{\circ},$$ 
   where $I_F$ denotes the inertia subgroup of the absolute Galois group of $F$. Since the formal degree conjecture is known for unipotent representations \cite{UnipotentFormalDegree}, we have 
   $$\fdeg(\pi_J)=\frac{\dim(\rho_{\pi_J})}{\abs{S^{\sharp}_{\varphi_{\pi_J}}}}\cdot \abs{\gamma(0,\Ad_{\mathcal{J}^{\vee}}\circ\varphi_{\pi_J},\psi)},$$ 
   where $(\varphi_{\pi_J},\rho_{\pi_J})$ is the enhanced $L$-parameter of $\mathcal{J}$ associated to $\pi_J$. 
   The proof of Conjecture \ref{HII} then boils down to a comparison between $\fdeg(\pi)$ and $\fdeg(\pi_J)$, together with of the dimensions of the enhancements, the component groups and the $\gamma$-factors of $\varphi_\pi$ and $\varphi_{\pi_J}.$

   We compare $\fdeg(\pi)$ and $\fdeg(\pi_J)$  by relating $\mathcal{H}(\mathfrak{s}^{\vee},q_F^{1/2})$ and $\mathcal{H}(\mathfrak{s}^{\vee}_J,q_F^{1/2}),$ where $\mathfrak{s}^{\vee}_J$ is an inertial equivalence class of $L$-parameters of $\mathcal{J}$ such that $\varphi_{\pi_J}\in\Phi(\mathcal{J})^{\mathfrak{s}^{\vee}_J}.$ We have $$\mathcal{H}(\mathfrak{s}^{\vee},q_F^{1/2}) \cong \mathcal{H}(\mathfrak{s}^{\vee}_J,q_F^{1/2})\rtimes\Gamma_{\mathfrak{s}^{\vee}}$$ for some finite group $\Gamma_{\mathfrak{s}^{\vee}}$, and the difference between $\fdeg(\pi)$ and $\fdeg(\pi_J)$ then consists only of a constant term, involving the volumes of certain types, together with an additional term arising from the fact that $\pi$ is obtained from $\pi_J$ via some induction procedure.
   
   On the other hand, we verify that the difference between the $\gamma$-factors of $\varphi_\pi$ and $\varphi_{\pi_J}$ is precisely the term corresponding to the volumes of these types, and the difference between the component groups and the dimension of the enhancements accounts for the remaining term.

    In the split case, modules over the Hecke algebra are classified by the so called Kazhdan–Lusztig triples \cite{PrinicipalDisconencted}, and the reduction can be carried out directly by construction. For quasi-split groups, the algebra $\mathcal{H}(\mathfrak{s}^{\vee},q_F^{1/2})$ may have unequal parameters. As a result, the classification of its modules requires a two-step reduction to graded Hecke algebras. This reduction was established by Lusztig \cite{AffineHeckeGradedVersionLusztig} for affine Hecke algebras, and later extended by Solleveld \cite[\S 2]{OnClassificationHeckeAlgebrasUnequalPar}, building on work of Opdam \cite[\S 4]{OnSpectralDecomposition}, to the extended case. To carry out our reduction to unipotent representations, in Section \ref{Reductions} we verify the compatibility of Lusztig’s procedure with induction from an affine Hecke algebra to an extended affine Hecke algebra.

   Section \ref{EpsilonSec} is dedicated to computing the absolute value of the $\varepsilon$-factors involved in the formal degree conjecture, using its inductivity in degree 0. Finally, in Section \ref{FD} we carry out the reduction to unipotent representations described above, and we prove Theorem \ref{FirstTheorem}.

    \paragraph{Acknowledgments}I would like to express my sincere gratitude to my supervisor and promoter, Prof. Maarten Solleveld, for suggesting this problem and for his guidance and support throughout the preparation of this paper. I am also grateful to Prof. Jessica Fintzen for pointing out a mistake in an earlier version of this work and for her helpful suggestions that improved the exposition.

\input{LusReduction}

\input{Type}

\input{Parameters}
\input{Gamma}

\input{FD}

\bibliographystyle{alpha}
\nocite{*}
\bibliography{Bibliography}

\end{document}

%% file: LusReduction.tex
\section{Lusztig's reduction and compatibility with finite induction}\label{Reductions}

In this section, we recall Lusztig's classification of irreducible modules of affine Hecke algebras \cite{AffineHeckeGradedVersionLusztig} in the form of \cite[\S 2]{OnClassificationHeckeAlgebrasUnequalPar} going through the reduction to graded Hecke algebras. We start by recalling some general information about (extended) affine Hecke algebras. This information can be found in many references, but we will follow \cite[\S 1]{OnClassificationHeckeAlgebrasUnequalPar} and keep their notation for the most part.

Let $\mathcal{R}=(X,R_0,Y,R_0^{\vee}, F_0)$ be any based root datum and denote by $W_0$ the Weyl group of $R_0.$ Let $$W=X\rtimes W_0\text{ and } W^{\aff}=\mathbb{Z}R_0\rtimes W_0.$$ and let $S^{\aff}$ be the set of affine simple reflections in $W^{\aff}.$ We usually denote by $t_x$ the translation in $W$ corresponding to $x\in X.$ Let $q:S^{\aff} \to \mathbb{C}^{\times}$ be such that 
$$q(s)=q(s') \text{ whenever } s \text{ and }s'\text{ are }W\text{-conjugate}.$$ 
We fix $q^{1/2}$ a square root of $q.$ We define the affine Hecke algebra $\mathcal{H}=\mathcal{H}(\mathcal{R},q)$ as the unique associative complex algebra with basis $\{N_w\mid w\in W\}$ and multiplication rules: \begin{itemize}
    \item $N_wN_v = N_{wv}$ if $l(wv)=l(w) + l(v)$,
    \item $(N_s-q(s)^{1/2})(N_s + q(s)^{1/2})=0$ if $s\in S^{\aff}.$ 
\end{itemize}

If $q(s)$ is always a power of some positive number $\mathbf{q}\in\mathbb{R}_{>0}$, $\mathcal{H}$ is sometimes denoted by $\mathcal{H}(\mathcal{R},\lambda,\lambda^*,\mathbf{q})$ with $$q(s_\alpha)=\mathbf{q}^{\lambda(\alpha)}\text{ and }q(s_\alpha')=\mathbf{q}^{\lambda^*(\alpha)}.$$

These algebras have been extensively studied, particularly because of their connection with representations of $p$-adic groups. For an introduction to affine Hecke algebras, we refer the reader to the survey paper \cite{SurveyHecke}. We will primarily be interested in an extended version of these algebras. Let $\Gamma$ a finite group of diagram automorphisms of $(R_0,F_0)$ such that $q(s_{\alpha})=q(s_{\gamma(\alpha)})$ for every $\gamma\in\Gamma$ for every $\alpha\in R_{\nr}:=R\cup\{2\alpha: \alpha^{\vee}\in 2Y\}.$ Then $\Gamma$ acts on $\mathcal{H}$ via isomorphisms as $$\gamma\cdot N_w=N_{\gamma(w)}.$$

and we form the crossed product algebra $\mathcal{H}\rtimes \Gamma.$

Denote by $T$ the complex torus $$T=\Hom_{\mathbb{Z}}(X,\mathbb{C}^{\times})=Y\otimes_{\mathbb{Z}} \mathbb{C}^{\times}.$$ Then we have $$Z(\mathcal{H})=\mathcal{O}(T)^{W_0}\text{ and }Z(\mathcal{H}\rtimes \Gamma)=\mathcal{O}(T)^{W_0\rtimes \Gamma},$$ where $\mathcal{O}(T)$ denotes the ring of regular functions on $T$ and $Z(A)$ denotes the center of $A.$  The complex torus $T$ has a polar decomposition $$T=T_{\rs}\times T_{\un}=\Hom_{\mathbb{Z}}(X,\mathbb{R}_{> 0})\times\Hom_{\mathbb{Z}}(X,S^1),$$ where $T_{\rs}$ is called the real split part and $T_{\un}$ is called the unitary part.

In Section \ref{ReductionSection} we will see that the study of modules over these algebras can be reduced to the study of modules over graded Hecke algebras. The latter are defined as a degeneration of affine Hecke algebras. Let $\mathfrak{a}=Y\otimes_{\mathbb{Z}}\mathbb{R}$ and $\mathfrak{a}^*=X\otimes_{\mathbb{Z}}\mathbb{R}$, and consider the degenerate root datum $$\tilde{\mathcal{R}}=(\mathfrak{a}, R_0,\mathfrak{a}^*,R^{\vee}_0,F_0).$$ Let $k:F_0\to \mathbb{C}$ be such that $k(\alpha)=k(\beta)$ if $\alpha$ and $\beta$ are $W_0$-conjugate. Then we can form the graded Hecke algebra $$\mathbb{H}=\mathbb{H}(\tilde{\mathcal{R}},k)=S(\mathfrak{t}^*)\otimes \mathbb{C}[W_0]$$ where $\mathfrak{t}^*=X\otimes_\mathbb{Z}\mathbb{C}$ and the multiplication is defined as follows: \begin{itemize}
    \item $S(\mathfrak{t}^*)$ and $\mathbb{C}[W_0]$ are embedded as subalgebras;
    \item For $x\in \mathfrak{t}^*$ and $s_\alpha\in S$ we have the cross relation  $$x\cdot s_\alpha-s_\alpha\cdot s_\alpha(x)=k(\alpha)\bigl\langle x,\alpha^{\vee}\bigl\rangle.$$
\end{itemize}

Again, if $\Gamma$ is a group of diagram automorphisms of $\tilde{\mathcal{R}}$ and $k(\alpha)=k(\gamma\cdot\alpha)$  for all $\alpha\in R_0$ and for all $\gamma\in \Gamma$, then $\Gamma$ acts on $\mathbb{H}$ via $$\gamma (x\cdot s_{\alpha})=\gamma(x)\cdot s_{\gamma(\alpha)}\text{ for all }x\in\mathfrak{t}^*,\alpha\in F_0,\gamma\in\Gamma.$$ 

The representations of specific graded Hecke algebras arising from $p$-adic groups were constructed geometrically by Lusztig \cite{CuspLocal1}. This construction was later extended to extended graded Hecke algebras in \cite{GradDisc}. We will come back to this later.

\subsection{Lusztig reductions}\label{ReductionSection}

We now explain Lusztig’s reduction from affine Hecke algebras to graded Hecke algebras. This was originally carried out in \cite{AffineHeckeGradedVersionLusztig}, and later generalized to affine Hecke algebras by Solleveld \cite[\S 2]{OnClassificationHeckeAlgebrasUnequalPar} building on an earlier work by Opdam \cite[\S 4]{OnSpectralDecomposition}. Lusztig considers a maximal ideal $I$ of $Z(\mathcal{H})$ and the $I$-adic completion of $\mathcal{H}$, in order to study representations of $\mathcal{H}$ that have central character corresponding to $I.$ In this way, there is no restriction on the central character and this could potentially lead to complications. This issue is addressed by the variation of Opdam, which replaces the central character with a neighborhood of it. Solleveld generalized all of this to extended affine Hecke algebras \cite[\S 2]{OnClassificationHeckeAlgebrasUnequalPar}, and we follow his work.

Let $t=uc\in T=T_{\un}T_{\rs}$ and consider the root system 
$$R_t=\{\alpha\in R_0: s_{\alpha}(t)=t\}.$$

The set of positive roots $R_0^+$ of $R_0$ defines a set of positive roots $R_t^+=R_0^+\cap R_t$ of $R_t.$ Let $F_t$ be the unique basis of $R_t$ contained in $R_0$ and consider the based root datum 
$$\mathcal{R}_t:=(X, R_t, Y, R_t^{\vee},F_t).$$

Consider the group $$W_{F_t,t}':=\{w\in W_0\rtimes \Gamma: w(t)=t,w(F_t)=F_t\}.$$ This is the complement of $W(R_t)$ in the isotropy group of $t$ in $W_0\rtimes \Gamma,$ meaning 
$$W_t':=(W_0\rtimes \Gamma)_t=W(R_t)\rtimes W_{F_t,t}'.$$
Notice that $F_t$ may not be contained in $F_0,$ so $\mathcal{H}_t:=\mathcal{H}(\mathcal{R}_t,q)$ is not necessarily a parabolic subalgebra of $\mathcal{H}.$

We now introduce the necessary analytic localization. Let $U\subset T$ be a non-empty $W_0\rtimes\Gamma$-invariant open subset, and let $C^{\an}(U)$ be the algebra of holomorphic functions on $U.$ Then we have a natural embedding 
$$Z(\mathcal{H}\rtimes\Gamma)\hookrightarrow C^{\an}(U)^{W_0\rtimes \Gamma},$$
and an algebra isomorphism $$C^{\an}(U)\cong \mathcal{O}(T)\otimes_{\mathcal{O}(T)^{W_0\rtimes\Gamma}}C^{\an}(U)^{W_0\rtimes\Gamma}.$$ So, we can construct the algebra $$\mathcal{H}^{\an}(U)\rtimes \Gamma:= C^{\an}(U)^{W_0\rtimes \Gamma}\otimes_{Z(\mathcal{H}\rtimes \Gamma)} (\mathcal{H}\rtimes \Gamma).$$

This algebra is introduced since, by \cite[Proposition 4.3]{OnSpectralDecomposition}, for any $U\subset T,$ the category $\Mod_{f,U}(\mathcal{H}\rtimes\Gamma)$ of finite dimensional left $\mathcal{H}\rtimes \Gamma$-modules whose $\mathcal{O}(T)$-weights lie in $U$ is naturally equivalent to $\Mod_f(\mathcal{H}^{\an}(U)\rtimes\Gamma)$, the category of left finite dimensional $\mathcal{H}^{\an}(U)\rtimes\Gamma$-modules.

To study $\mathcal{H}$-representations with central character $(W_0\rtimes\Gamma)t$, a well-behaved $(W_0\rtimes\Gamma)_u$-stable $U_u\subset T$ of $t$ is constructed in \cite[\S 2.1]{OnClassificationHeckeAlgebrasUnequalPar}. Let $U=(W_0\rtimes \Gamma)U_u.$ Then, the relevant algebras for us are going to be $\mathcal{H}^{\an}(U)\rtimes \Gamma$ and $\mathcal{H}_u^{\an}(U_u)\rtimes W_{F_u,u}':=\mathcal{H}(R_u,q_u)^{\an}(U_u)\rtimes W_{F_u,u}'.$

\begin{theorem}\label{Reduction1}\cite[Theorem 2.1.2(c)]{OnClassificationHeckeAlgebrasUnequalPar},\cite[Theorem 4.10]{OnSpectralDecomposition},\cite[Theorem 8.6]{AffineHeckeGradedVersionLusztig} 
There exists an explicit idempotent $e\in C^{\an}(U)^{W_0\rtimes\Gamma}$ and an isomorphism of $C^{\an}(U)^{W_0\rtimes \Gamma}$-algebras $$e(\mathcal{H}^{\an}(U)\rtimes\Gamma)e\cong \mathcal{H}_u^{\an}(U_u)\rtimes W_{F_u,u}',$$ which induces an isomorphism
$$\mathcal{H}^{\an}(U)\rtimes \Gamma\cong M_{[W_0\rtimes\Gamma; (W_0\rtimes\gamma)_u ]}[\mathcal{H}_u^{\an}(U_u)\rtimes W_{F_u,u}'],
$$
where $M_n[A]$ denotes the algebra of $n\times n$ matrices with coefficients in $A.$ 
This induces an equivalences of categories: 
$$\Mod_{f,U}(\mathcal{H}(R,q)\rtimes\Gamma) \xrightarrow[]{L_1} \Mod_{f,U_u}(\mathcal{H}_u\rtimes W_{F_u,u}').$$
\end{theorem}

To reduce our study to the case of graded Hecke algebras, we use a notion of analytic localization for graded Hecke algebras similar to the one for affine Hecke algebras. Namely, let $\mathbb{H}\rtimes \Gamma$ be an extended graded Hecke algebra associated with $\widetilde{\mathcal{R}},$ and let $V\subset \mathfrak{t}$ a $W_0\rtimes \Gamma$-invariant open subset. Then we can form the algebra 
$$\mathbb{H}^{\an}(V)\rtimes \Gamma:=C^{\an}(V)^{W_0\rtimes \Gamma}\otimes_{Z(\mathbb{H}\rtimes \Gamma)}(\mathbb{H}\rtimes \Gamma),$$
which has analogous properties to the analytic localization for affine Hecke algebras. Define a parameter function $k_u$ for the degenerate root datum $\tilde{\mathcal{R}_u}=(\mathfrak{a},R_u,\mathfrak{a}^*,R_u^{\vee},F_u)$ by 
$$k_u(\alpha)=(\log q(s_\alpha) +\alpha(u)\log q(t_\alpha s_\alpha))/2.$$ 
Let $u\in T_{\un}$ and consider the exponential map $\exp_u:\mathfrak{t}\to T$ for $T$ based at $u,$ which sends $y\mapsto u\exp(y).$

\begin{theorem}\label{Reduction2}\cite[Theorem 2.1.4]{OnClassificationHeckeAlgebrasUnequalPar},\cite[Theorem 9.3]{AffineHeckeGradedVersionLusztig} There is an isomorphism of algebras $$\Phi_u:\mathcal{H}^{\an}(\exp_u(V))\to \mathbb{H}(\tilde{\mathcal{R}}_u, ku)^{\an}(V),$$ 
which extends to an isomorphism 
$$\Phi_u:\mathcal{H}^{\an}(\exp_u(V))\rtimes\Gamma\to \mathbb{H}(\tilde{\mathcal{R}}_u, ku)^{\an}(V)\rtimes\Gamma.$$ 
This gives an equivalence of categories $$\Mod_{f, W_u't}(\mathcal{H}(\mathcal{R}_u,q_u)\rtimes\Gamma)\xrightarrow[]{L_2}\Mod_{(W(R_u)\rtimes W_{F_u,u}')\log(c)}(\mathbb{H}(\tilde{\mathcal{R}}, k_u)\rtimes W_{F_u,u}').$$
\end{theorem}

Graded Hecke algebras are much easier to study than affine Hecke algebras. The modules over some specific extended graded Hecke algebras of geometric nature were studied in \cite{GradDisc}, and the classification can be interpreted in terms of $L$-parameters. Let $M$ be any (possibly disconnected) complex reductive group. Let $L$ be a Levi subgroup of $M^{\circ}$ and let $v\in\Lie(L)$ be nilpotent. Let $\mathcal{C}_v^L$ be the adjoint orbit of $v$ and let $\mathcal{L}$ be an irreducible $L$-equivariant cuspidal local system on $\mathcal{C}_v^L$. As defined in \cite{GeneralizationSpringer} we call $(L,\mathcal{C}_v^L,\mathcal{L})$ a cuspidal support for $M$. Let $\mathbb{H}(M,L,\mathcal{L})$ be the twisted graded Hecke algebra defined in \cite[Proposition 2.2]{GradDisc} associated to such a cuspidal support. Then, modules over this graded Hecke algebras were classified:

\begin{theorem}\label{discclas}\cite[Theorem 3.20]{GradDisc}  There exists a canonical bijection between the set $\Irr(\mathbb{H}(M,L,\mathcal{L}))$ and the set of $M$-conjugacy classes of triples $(y,\sigma,\rho)$ where $y\in \Lie(M^{\circ})$ is nilpotent, $\sigma\in \Lie(M^{\circ})$ is semisimple such that $[\sigma,y]=0,$ $\rho\in \pi_0(Z_M(\sigma,y))$ and $(Z_M(\sigma), y,\rho)$ has cuspidal support $(L,\mathcal{C}_v^L,\mathcal{L})$ up to $M$-conjugation.    
\end{theorem}

\subsection{Compatibility with finite induction}

We now explain why the maps from Theorem \ref{Reduction1}, \ref{Reduction2} and \ref{discclas} are compatible with finite induction from an affine Hecke algebra to an extended affine Hecke algebra. We will denote the category of finite dimensional left modules over some algebra $A$ as $A\text{-}\Mod$. Then we check that the following lemma holds: 

\begin{lemma}
    Lusztig's reduction process on extended affine Hecke algebras is compatible with induction on $\Gamma.$ In other words, the following diagram commutes

\begin{equation}\label{Diagram1}
\begin{tikzcd}
	{\mathcal{H}\textnormal{-}\Mod} & {\mathcal{H}(\mathcal{R}_u,q_u)\rtimes W_{F_u,u}\textnormal{-}\Mod} & {\mathbb{H}_u\rtimes W_{F_u,u}\textnormal{-}\Mod} \\
	{\mathcal{H}\rtimes \Gamma\textnormal{-}\Mod} & {\mathcal{H}(\mathcal{R}_u,q_u)\rtimes W_{F_u,u}'\textnormal{-}\Mod} & {\mathbb{H}_u\rtimes W_{F_u,u}'\textnormal{-}\Mod}
	\arrow["{L_1}", squiggly, from=1-1, to=1-2]
	\arrow["\Ind", from=1-1, to=2-1]
	\arrow["{L_2}", squiggly, from=1-2, to=1-3]
	\arrow["\Ind", from=1-2, to=2-2]
	\arrow["\Ind", from=1-3, to=2-3]
	\arrow["{L_1}", squiggly, from=2-1, to=2-2]
	\arrow["{L_2}", squiggly, from=2-2, to=2-3]
\end{tikzcd}
\end{equation}
    where $L_1$ and $L_2$ denote the two steps of Lusztig's reduction process obtained from Theorem \ref{Reduction1} and Theorem \ref{Reduction2} respectively, and $W_{F_u,u}$ denotes the set of elements in $W_0$ which fix $u$ and stabilize $F_u.$ 
\end{lemma}

\begin{proof}

Zooming in on the left square, we have

\begin{equation}\label{Square1}
\begin{tikzcd}[column sep=small]
	{\mathcal{H}\text{-}\Mod} & {\mathcal{H}^{\an}(U)\text{-}\Mod} & {M_{[W_0:W_u]}(e\mathcal{H}^{\an}(U)e)\text{-}\Mod} \\
	{\mathcal{H}\rtimes\Gamma\text{-}\Mod} & {\mathcal{H}^{\an}(U)\rtimes\Gamma\text{-}\Mod} & {M_{[W_0\rtimes \Gamma: (W_0\rtimes \Gamma)_u]}(e(\mathcal{H}^{\an}(U)\rtimes \Gamma )e)\text{-}\Mod}
	\arrow[squiggly, from=1-1, to=1-2]
	\arrow["\Ind", from=1-1, to=2-1]
	\arrow["\cong", from=1-2, to=1-3]
	\arrow["\Ind", from=1-2, to=2-2]
	\arrow["\Ind", from=1-3, to=2-3]
	\arrow[squiggly, from=2-1, to=2-2]
	\arrow["\cong", from=2-2, to=2-3]
\end{tikzcd}
\end{equation}

\begin{equation}\label{Square2}
\begin{tikzcd}[column sep=small]
	{M_{[W_0:W_u]}(e\mathcal{H}^{\an}(U)e)\text{-}\Mod} & {M_{[W_0: W_u]}(\mathcal{H}_u^{\an}(U_u)\rtimes W_{F_u,u})\text{-}\Mod} \\
	{M_{[W_0\rtimes \Gamma: (W_0\rtimes \Gamma)_u]}(e(\mathcal{H}^{\an}(U)\rtimes \Gamma )e)\text{-}\Mod} & {M_{[W_0\rtimes \Gamma: (W_0\rtimes \Gamma)_u]}(\mathcal{H}_u^{\an}(U_u)\rtimes W_{F_u,u}')\text{-}\Mod}
	\arrow["\cong", from=1-1, to=1-2]
	\arrow["\Ind", from=1-1, to=2-1]
	\arrow["\Ind", from=1-2, to=2-2]
	\arrow["\cong", from=2-1, to=2-2]
\end{tikzcd}
\end{equation}

\begin{equation}\label{Square3}    
\begin{tikzcd}[column sep=small]
	{M_{[W_0: W_u]}(\mathcal{H}_u^{\an}(U_u)\rtimes W_{F_u,u})\text{-}\Mod} & {\mathcal{H}_u^{\an}(U_u)\rtimes W_{F_u,u}\text{-}\Mod} & \\
	{M_{[W_0\rtimes \Gamma: (W_0\rtimes \Gamma)_u]}(\mathcal{H}_u^{\an}(U_u)\rtimes W_{F_u,u}')\text{-}\Mod} & {\mathcal{H}_u^{\an}(U_u)\rtimes W_{F_u,u}'\text{-}\Mod} & 
	\arrow[dashed, from=1-1, to=1-2]
	\arrow["\Ind", from=1-1, to=2-1]
	\arrow["\Ind", from=1-2, to=2-2]
	\arrow[dashed, from=2-1, to=2-2]
\end{tikzcd}
\end{equation}

\begin{equation}\label{Square4}
    \begin{tikzcd}
	{\mathcal{H}_u^{\an}(U_u)\rtimes W_{F_u,u}} & {\mathcal{H}_u\rtimes W_{F_u,u}} \\
	{\mathcal{H}_u^{\an}(U_u)\rtimes W_{F_u,u}'} & {\mathcal{H}_u\rtimes W_{F_u,u}'}
	\arrow[squiggly, from=1-1, to=1-2]
	\arrow["\Ind", from=1-1, to=2-1]
	\arrow["\Ind", from=1-2, to=2-2]
	\arrow[squiggly, from=2-1, to=2-2]
\end{tikzcd}
\end{equation}

The arrows $\rightsquigarrow$ denote the maps from $\Mod_{f,U}(\mathcal{H}\rtimes\Gamma)$ to $\Mod_f(\mathcal{H}^{\an}(U)\rtimes\Gamma),$ and the maps $\dashrightarrow$ denote the Morita equivalence between $M_n[A]$ and $A$ for every algebra $A.$

The left square of \ref{Square1} commutes since the Morita equivalence from Theorem \ref{Reduction1} just extends the action from $\mathcal{H}$ to $\mathcal{H}^{\an}(U).$ The right square of \ref{Square1} and the square \ref{Square2} commute since the vertical arrows are all inductions from one algebra to another up to isomorphism. 

In \ref{Square3}, the square commutes since the Morita equivalence between $M_n[A]$ and $A$ is compatible with group actions. More generally, whenever there is an inclusion of algebras $A\hookrightarrow B,$ it is clear that for every $n,d>0$ the following diagram commutes: 
\[\begin{tikzcd}
	{M_d(A)\text{-}\Mod} & {A\text{-}\Mod} & {M_n(A)\text{-}\Mod} \\
	{M_d(B)\text{-}\Mod} & {B\text{-}\Mod} & {M_n(B)\text{-}\Mod}
	\arrow[dashed, from=1-1, to=1-2]
	\arrow["\Ind", from=1-1, to=2-1]
	\arrow["\Ind", from=1-2, to=2-2]
	\arrow[dashed, from=1-3, to=1-2]
	\arrow["\Ind", from=1-3, to=2-3]
	\arrow[dashed, from=2-1, to=2-2]
	\arrow[dashed, from=2-3, to=2-2]
\end{tikzcd}\]

Finally,  \ref{Square4} commutes for the same reason as the left square of \ref{Square1}.

One can check that the arguments for the right square of \ref{Diagram1} are similar, and we conclude. \qedhere \end{proof}

On the other hand, Theorem \ref{discclas} is built on \cite[Lemma 3.13 and Lemma 3.16]{GradDisc}, which has already established compatibility with induction. In fact, the following is proved: 

\begin{lemma}\cite[Lemma 3.13 and Lemma 3.16]{GradDisc} We keep the set up of Theorem \ref{discclas}. Let $\rho^\circ\in \Irr(\pi_0(Z_{M^{\circ}}(\sigma,y)))$ and let $\pi_{\rho^{\circ}}$ be the associated $\mathbb{H}(M^{\circ},L,\mathcal{L})$-module. Then, there is a bijection between the set

$$\{\rho\in\Irr(\pi_0(Z_M(\sigma,y)))\mid \rho|_{\pi_0(Z_{M^{\circ}}(\sigma,y))}\text{ contains }\rho^{\circ}\}$$
and the set $$\{\pi_\rho\in\Irr(\mathbb{H}(M,L,\mathcal{L})\mid \pi_\rho|_{\mathbb{H}(M^{\circ},L,\mathcal{L})}\text{ contains }\pi_{\rho^{\circ}}\}.$$

The map from the first set to the second is the following: let $\pi_0(Z_M(\sigma,y))_{\rho^\circ}$ be the stabilizer of $\rho^{\circ}$ in $\pi_0(Z_M(\sigma,y))$ and let $\tau\in \Irr(\pi_0(Z_M(\sigma,y))_{\rho^\circ}).$ Then, $\tau$ is a representation of the stabilizer $\mathbb{H}(M,L,\mathcal{L})_{\pi_{\rho^{\circ}}}$ of $\pi_{\rho^{\circ}}$ in $\mathbb{H}(M,L,\mathcal{L})$, and the map is $$\ind_{\pi_0(Z_M(\sigma,y))_{\rho^\circ}}^{\pi_0(Z_M(\sigma,y))}(\tau\otimes\rho^{\circ})\mapsto \ind_{\mathbb{H}(M,L,\mathcal{L})_{\pi_{\rho^{\circ}}}}^{\mathbb{H}(M,L,\mathcal{L})}(\tau^*\otimes \pi_{\rho^{\circ}}),$$
where $\tau^*$ is the dual of $\tau.$
    
\end{lemma}

%% file: Type.tex
\section{Types for principal series}\label{TypeSection}

We now move to the realm of $p$-adic groups, and we set the notation for the remainder of this paper. Let $G=\mathcal{G}(F)$ be the $F$-points of a quasi-split connected reductive group over $F$ which splits over an extension $E/F$ with ramification index $e=e(E/F).$ Let $T=\mathcal{T}(F)\subset G$ be a maximal torus containing a maximal split torus $S=\mathcal{S}(F).$ In general, we will always denote with calligraphic letters the scheme, and with normal letters the $F$-points of the scheme. 

Let $B$ be a fixed Borel subgroup containing $T$, and let $\chi:T\to \mathbb{C}^{\times}$ be a character of $T$. We will denote by $\Rep(G)$ the category of smooth $G$-representations on complex vector spaces, and by $\Irr(G)$ the set of equivalence classes of irreducible objects in $\Rep(G)$. We will denote by $\mathcal{H}(G)$ the Hecke algebra of $G$, which is the convolution algebra of compactly supported, locally constant functions $f:G\to \mathbb{C}$. We recall that there is a well-known equivalence between $\Rep(G)$ and the category of non-degenerate modules over $\mathcal{H}(G)$. 

Let $\mathfrak{s}=[T,\chi]_G$ be the inertial equivalence class of $G$ associated with $\chi$ and let $\Irr(G)^{\mathfrak{s}}$ be the associated Bernstein component. We recall that $\Irr(G)^{\mathfrak{s}}$ consists of all irreducible subquotients of the normalized parabolic inductions $I_B^G(\chi')=\ind_B^G(\chi'\otimes\delta_B^{1/2})$ where $\chi'\in [T,\chi]_T$ and $\delta_B$ is the modulus character.  

In this section, we will construct a type for $\mathfrak{s}.$ We recall the definition of a type for $\mathfrak{s}$: 

\begin{definition}
    A type $t=(J,\sigma)$ for $\mathfrak{s}$ is a pair consisting of a compact open subgroup $J\subset G$ and an irreducible representation $\sigma$ of $J,$ satisfying the following:
    
    Let $e_t\in\mathcal{H}(G)$ be the idempotent defined by  
    $$e_t(x)=\begin{cases}\frac{\Tr(\sigma(x^{-1}))}{\vol(J)}\text{ if } x\in J;\\
    0\hspace{1.5cm}\text{ else.} 
    \end{cases}$$
    Let $\Rep(G)_t$ be the full subcategory of $\Rep(G)$ consisting of representations $(\pi, V)$ such that $\mathcal{H}(G)(e_tV)=V$, and let $\mathcal{H}_t=e_t\mathcal{H}(G)e_t$. Then the functor $\Rep(G)_t\to \mathcal{H}_t\text{-}\Mod$ given by $V\mapsto e_t(V)$ is an equivalence of categories, and $\Rep(G)^{\mathfrak{s}}=\Rep(G)_t$ are equal as full subcategories of $\Rep(G).$
\end{definition}

One of the biggest difficulties in working with principal series of quasi-split groups is that we do not have a type for them, and the machinery of \cite{RochePrincipal} does not 
generalize straightforwardly. Types for the unitary group $U_n$ for every $n$ and for any degree 2 extension $F'/F$ whenever $p\neq 2,$ were constructed using the machinery of strata developed by Miyauchi and Stevens \cite{MyauchiStevens}. These types look very similar to the ones obtained by Roche, since they are of the form 
$$J_f=\bigl\langle T^0,U_{\alpha,x,f}\bigl\rangle$$
with $f$ a concave function depending only on the depth of $\chi\circ \alpha^{\vee}$ for every absolute root $\alpha,$ and $x$ a point in the building. We recall that $f:R(\mathcal{G},\mathcal{S})\cup \{0\}\to\mathbb{R}\cup\{\infty\}$ is called concave if 
$$f(\alpha+\beta)\leq f(\alpha) +f(\beta)$$ for every $\alpha,\beta\in R(\mathcal{G},\mathcal{S})\cup\{0\},$ with $\alpha+\beta\in R(\mathcal{G},\mathcal{S})\cup \{0\}.$

We will now directly define a type $J_f$ for $[T,\chi]_G$. Let $R^+=R(\mathcal{G},\mathcal{S})^+:=\{\alpha_1,\dots,\alpha_n\}$ and $R^-=R(\mathcal{G},\mathcal{S})^-$ be the sets of positive and negative relative roots of $R=R(\mathcal{G},\mathcal{S})$ with respect to the basis $\Delta$ defined by $B$. For a root $\alpha\in R(\mathcal{G},\mathcal{S})$ we will denote by $\tilde{\alpha}$ a fixed preimage of $\alpha$ in $R(\mathcal{G},\mathcal{T})$ and by $F_{\alpha}$ the splitting field of $\tilde{\alpha}.$ We denote by 
$$\alpha^{\vee}=\prod_{\sigma\in\Gal(F_{\alpha}/F)}{}^{\sigma}\tilde{\alpha}^{\vee}:F_{\alpha}^{\times}\to T(F).$$ 
Let $c_{\alpha}$ denote the conductor of the character $\chi\circ  \alpha^{\vee}: F_{\alpha}^{\times}\to \mathbb{C}^{\times}.$ We define $f:R\cup\{0\}\to\mathbb{R}$ as the function 

\begin{equation}\label{function}
    f(\alpha):=\begin{cases}
    e^{-1}\left \lfloor \dfrac{e\cdot c_{\alpha}}{2} \right\rfloor & \text{if } \alpha \in R^+, \\[6pt]
    \max\big\{e^{-1},e^{-1}\left\lfloor \dfrac{(e\cdot c_{\alpha}) + 1}{2} \right\rfloor\big\} & \text{if } \alpha \in R^-, \\[6pt] 0 & \text{if }\alpha=0.
\end{cases}
\end{equation}

We remark that the conductor $c_\alpha$ lies in the image of the valuation, which is $\frac{1}{e}\mathbb{Z}$. That's why, in the definition of $f$ we multiply and divide by $e$, so that $e\cdot c_\alpha$ is always an integer. Moreover, we notice that if $c_\alpha>0,$ then $f(\alpha)+f(-\alpha)=c_\alpha.$ 

\begin{example}\label{ExampleU4}

    Consider the group $U_4=U_4(E/F)$ for any degree 2 extension $E/F$. Assume that the characteristic of the residue field $\mathfrak{f}$ is not 2. We have a maximal torus $T$ and a maximal split torus $S\subset T$ given by 
    $$T=\mqty[\dmat{a,b,\overline{b}^{-1},\overline{a}^{-1}}]\text{ with }a,b\in E^{\times},$$
    $$S=\mqty[\dmat{a,b,b^{-1},a^{-1}}]\text{ with }a,b\in F^{\times}.$$

   We fix any smooth character $\chi$ of $T$. Consider the relative root $\alpha: (a,b)\mapsto ab^{-1}.$ Then, its preimage under the restriction map is given by the two absolute roots $\alpha_1:(a,b)\mapsto ab^{-1}$ and $\alpha_2:(a,b)\mapsto \overline{a}\overline{b}^{-1}.$ On the other hand, the root $\beta\in R(\mathcal{G},\mathcal{S})$ given by $\beta:(a,b)\to b^{2}$ is already split.

   Then $\alpha^{\vee}: E^{\times}=F_{\alpha}^{\times}\to T$ is given by 
   
   $$x\mapsto \mqty[\dmat{x,\overline{x},\overline{x}^{-1},x^{-1}}]$$ while $\beta^{\vee}:F^{\times}\to T$ given by  
   
   $$x\mapsto \mqty[\dmat{1,x,x^{-1}, 1}].$$ For the sake of comparing the conductor of $\chi\circ\beta^{\vee}$, $\chi\circ\alpha^{\vee}$ and $\chi\circ(\alpha+\beta)^{\vee},$ we notice that we can extend $\beta^{\vee}$ to $F_{\alpha}^{\times}$ by first base changing $\beta^{\vee}$ to $F_{\alpha}^{\times}$ and then composing with the norm map $N_{T,F_\alpha/F}$ on $T(F_{\alpha}).$ Since the norm map acts as a square on the image of $\beta^{\vee},$ and since $\textnormal{char}(\mathfrak{f})\neq 2,$ the conductor of $\chi$ composed with this extension of $\beta^{\vee}$ to $F_{\alpha}^{\times}$ is the same as the conductor of $\chi\circ\beta^{\vee}.$ Comparing the various conductors now consists of the same computation as in the case of split groups \cite[Lemma 3.4]{RochePrincipal}: $$\cond (\chi\circ\alpha^{\vee})+\cond (\chi\circ\beta^{\vee})\geq \cond (\chi\circ (\alpha+\beta)^{\vee}).$$

\end{example}

\begin{lemma}
    Assume that the characteristic of the residue field $\mathfrak{f}$ of $F$ is not $2$ or $3.$ Then the function $f$ defined by \ref{function} is concave.
\end{lemma}

\begin{proof}
    First we assume that $\mathcal{G}$ is absolutely simple, so that $R(\mathcal{G},\mathcal{T})$ is irreducible. Let $\alpha,\beta \in R(\mathcal{G},\mathcal{S})$, and let $\tilde{\alpha}, \tilde{\beta} \in R(\mathcal{G},\mathcal{T})$ be chosen preimages of $\alpha$ and $\beta$, respectively. Then we can assume, without loss of generality, that $F_{\beta}\subseteq F_{\alpha}.$ In fact, by just looking at all the possible actions of the Galois group on the Dynkin diagram, we can see that there is no case in which two roots split over two disjoint extensions.

     So we can base change $\beta^{\vee}:F_{\beta}^{\times}\to T$ to $F_{\alpha}$ and composing with the norm map $N_{T,F_{\alpha}/F}$ we obtain $$\beta^{\vee}_{F_{\alpha}}: F_{\alpha}^{\times}\to T(F_{\alpha})\xrightarrow[]{N_{T,F_{\alpha}/F}} T.$$

    The norm $N_{T,F_{\alpha}/F}$ on the image of $\beta^{\vee}$ acts as a square, a cube, or a power of six depending on the root data of $G$ and on $\alpha$ and $\beta$. Since the characteristic of $\mathfrak{f}$ is not 2 or 3, the conductor of $\chi\circ\beta^{\vee}$ is the same as the conductor of $\chi\circ\beta^{\vee}_{F{\alpha}}.$ Now, $\chi\circ\beta^{\vee}_{F_{\alpha}}$ and $\chi\circ\alpha^{\vee}$ are both defined on $F_{\alpha}^{\times},$ and we can use the same argument as in \cite[Lemma 3.4]{RochePrincipal} to conclude that $f$ is concave. In fact, we can say that there exists $l,m,n\in\mathbb{Z}$ such that 
    $$l(\alpha^{\vee}+\beta_{F_{\alpha}}^{\vee})=m\alpha^{\vee}+n\beta_{F_\alpha}^{\vee},$$
    and since the map $x\mapsto x^s$ is invertible on $1+\mathfrak{o}_F^i$ ($i\geq 1$) for $s$ prime to the characteristic of the residue field, we conclude that $c_{\alpha+\beta}\leq \max(c_\alpha,c_\beta).$ This condition implies that $f$ is concave. For example, if $\alpha>0$, $\beta<0$, $c_\beta\geq c_\alpha$ and $\alpha+\beta<0$ then 
    $$f(\alpha+\beta)=e^{-1}\left\lfloor\frac{(e\cdot c_{\alpha+\beta})+1}{2}\right\rfloor\leq e^{-1}\left\lfloor\frac{(e\cdot c_{\beta})+1}{2}\right\rfloor\leq e^{-1}\left\lfloor\frac{e\cdot c_{\alpha}}{2}\right\rfloor+e^{-1}\left\lfloor\frac{(e\cdot c_{\beta})+1}{2}\right\rfloor=f(\alpha)+f(\beta).$$ One can study all other cases similarly.

    Notice that if $R(\mathcal{G},\mathcal{S})$ is non-reduced, this does not create problems, since in our proof nothing breaks if one chooses $\beta$ to be $\alpha$. Moreover, if $R(\mathcal{G},\mathcal{S})$ is irreducible and non reduced, then $2\alpha+\beta$ is never a root if $\beta$ is not a multiple of $\alpha.$
    
    If $R(\mathcal{G},\mathcal{T})$ is not irreducible, then the Weil group might act by permuting its irreducible components. If $\tilde{\alpha}$ and $\tilde{\beta}$ belong to two different irreducible components that are not in the same orbits for the action of the Weil group, then $\alpha+\beta$ is not a root and the condition on the conductors is trivial. On the other hand, if $\tilde{\alpha}$ and $\tilde{\beta}$ belong to two different irreducible components which are in the same orbit for the action of the Weil group, then $\tilde{\beta}$ is conjugate to a root $\gamma$ in the same irreducible component as $\tilde{\alpha}.$ Then 
    $$\cond(\chi\circ\beta^{\vee}) =\cond(\chi\circ\gamma^{\vee}),$$ 
    and from the case of $R(\mathcal{G},\mathcal{T})$ irreducible, we conclude. 
    \end{proof}

\begin{remark}
    For the remainder of this paper, we assume that the characteristic of $\mathfrak{f}$ is neither 2 nor 3, so that $f$ is concave.
\end{remark}

Let $x$ be the barycenter of the chamber associated with the Iwahori subgroup $I$ coming from our choice of a Borel $B$ and let $T^0$ be the maximal compact subgroup of $T$. Since $f$ is concave, we can consider the compact subgroup $J_{x,f}$ associated with $f$ and $x$. We recall that this is the subgroup generated by $T^0$ and by $U_{\alpha,x,f}=U_{\alpha,x,f(\alpha)}U_{\alpha,x,f(2\alpha)}$ for every $\alpha\in R(\mathcal{G},\mathcal{S}).$ We will usually suppress the $x$ in the notation and just denote it $J_f, U_{\alpha,f}$ or $ U_{\alpha,f(\alpha)}.$ For the definition of these groups, as well as the higher Moy--Prasad filtration subgroups $T_r\subset T^0$ for $r\geq 0$, we refer to \cite[Chapter 7]{KalPras}.

The compact subgroup $J_f$ is our candidate for a type for $[T,\chi]_G.$ In order to prove that this is a type, we will use some results from \cite{Blondel1,Blondel2}. We start by recalling the following characterization of a cover of a type from \cite{Blondel2}. Let $J\subset G$ be a compact open subgroup with an irreducible smooth representation $\chi_0$. Write $B=TU^+$ and let $J^+=J\cap U^+$ and $J^-=J\cap U^-$ where $U^-$ is the unipotent subgroup of the opposite Borel $B^-.$ 

\begin{proposition}\cite[Corollaire, page 255]{Blondel2}\label{Blondel2} The pair $(J,\chi_0)$ is a type for $[T,\chi]_G$ if: \begin{enumerate}
    \item $J\cap T=T^0$ and $J=J^+T^0J^-$;
    \item The restriction of $\chi_0$ to $J^\pm$ is trivial, and $\chi_0|_{T^0}=\chi.$
    \item For every irreducible smooth representation $(\pi,V)$ of $G$, the Jacquet restriction $$r_{U^+}: V^{\chi_0}\to (V_{U^+})^{\chi}$$  is injective.
\end{enumerate}  

\end{proposition}

Notice that
$$\alpha^{\vee}(F_{\alpha}^{\times})\cap \bigl\langle U_{\alpha,f}, U_{-\alpha,f}\bigl\rangle\subset \alpha^{\vee}(F_{\alpha}^{\times})\cap T_{f(\alpha)+f(-\alpha)}\subset\alpha^{\vee}(F_{\alpha}^{\times})\cap T_{c_{\alpha}},$$ 
therefore
$$T\cap \bigl\langle J_f^+, J_f^-\bigl\rangle\subset \ker(\chi).$$
So, if we set $J=J_f$ we can extend $\chi$ to $J_f$ so that it is trivial on $J_f^+$ and $J_f^-$. We call this extension $\chi_0.$ Then the first 2 conditions are obviously true by construction. The third condition will come from a modified version of the following theorem:

\begin{theorem}\label{blondel}\cite[Théorème 1]{Blondel1}
Let $$(J_i^+)_{i \ge 0} \quad \text{and} \quad (J_i^-)_{i \ge 0}$$ be sequences of compact open subgroups of $U^+$ and $U^-$ respectively, satisfying the following conditions:

\begin{enumerate}
\item For every $i \ge 0$, the product $$J_i = J_i^-\, T^0\, J_i^+$$ is a group.

\item For every $i \ge 0$, the representation $\chi$ extends to a representation $\chi_i$ of $J_i$, 
trivial on $J_i^+$ and $J_i^-$, and equal to $\chi$ on $T^0$.

\item The sequence $(J_i^+)_{i \geq 0}$ is increasing with $$\bigcup_{i \geq 0} J_i^+ = U^+,$$ and the sequence $(J_i^-)_{i \geq 0}$ is decreasing.

\item For every $i \ge 0$ and every element $x \in J_{i+1}^+ \setminus J_i^+$, there exists an open subgroup $H_{i,x}$ of $J_i^-$ such that the conjugate $x H_{i,x} x^{-1}$ is a subgroup of $J_i$ having no nonzero fixed vector in the representation $\chi_i$.
\end{enumerate}

Then, for every smooth representation $(\pi, V)$ of $G$ and every $i \geq 0$, 
the restriction of the Jacquet restriction
$$r_{U^+}:V\longrightarrow V_{U^+}$$
to the $\chi_i$-isotypic component $V^{\chi_i}$ is an injective map
$$r_{U^+}|_{V^{\chi_i}} : V^{\chi_i} \hookrightarrow (V_{U^+})^{\chi}.$$
\end{theorem}
As one can see, finding Iwahori types is not possible using this theorem. In fact, Condition 4 requires $\chi$ to not have any non-zero fixed vector on some $H_{i,x}$, which is not possible if $\chi$ is trivial on $J\cap T^0.$ But in order to prove that $r_{U^+}$ is injective, the hypothesis can be relaxed. 

\begin{theorem}\label{type}
    The pair $(J_f,\chi_0)$ is a type for $[T,\chi]_G.$
\end{theorem}

Since we want to apply the same strategy as in Theorem \ref{blondel}, we will start by defining a sequence of compact subgroups $\{J_{f_i}\}_{i\geq 0}.$ Let $\mathcal{A}=\mathcal{A}(T)=\mathcal{A}(S)$ be the apartment of the Bruhat-Tits building associated with $T$. We choose a sequence of points $\{x_k\}_{k\geq 0}\in \mathcal{A}$ in the following way: we fix a special vertex $y$ in $\mathcal{A}$ and consider the negative cone with vertex $y,$ i.e. the set of elements $$\overline{C^-}:=y+C^-:=y+\{y'\in\mathcal{A}\mid \big\langle\lambda,y'\bigl\rangle<0\text{ for every }\lambda\in \Delta\}.$$ Let $\mathcal{C}$ be the only alcove in $\overline{C^-}$ containing $y$, and consider the barycenter $x_0$ of $\mathcal{C}$. Because the building is a chamber complex and $\overline{C^-}$ is convex, any alcove in $\overline{C^-}$ can be reached from $\mathcal{C}$ by a gallery all contained in $\overline{C^-}$. Now we choose a sequence $\{x_k\}_{k\geq 1}\subset\mathcal{A}(T)$ such that:\begin{itemize}
    \item  $x_k$ is the barycenter of some alcove $\mathcal{C}_k$ such that $\mathcal{C}_k\neq \mathcal{C}_{k'}$ for $k\neq k',$ and $\mathcal{C}_{k}\subset \overline{C^-}$ for every $k\geq 1.$
    \item $\mathcal{C}_k$ is obtained from $\mathcal{C}_{k-1}$ using just one reflection, and the intersection $\mathcal{C}_k\cap C_{k-1}$ is a facet of maximal dimension $-1$. 
    \item If $\mathcal{C}_k=s_{\alpha+l}\mathcal{C}_{k-1}$ for some root $\alpha$ and some integer $l,$ then $\mathcal{C}_{k'}\neq s_{\alpha+l}\mathcal{C}_{k'-1}$ for every $k'\neq k.$ In other words, we can reflect through a fixed hyperplane just once.  
    \item For every root $\alpha\in R^+,$ there are infinitely many chambers $\mathcal{C}_k$ such that $\mathcal{C}_k=s_{\alpha+l}\mathcal{C}_{k-1}$ for some integer $l.$
\end{itemize}

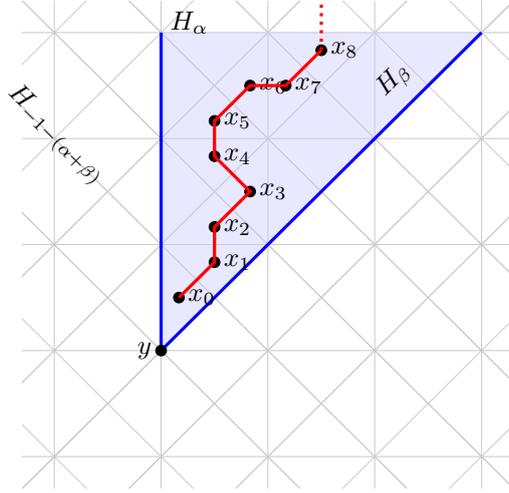
\begin{figure}[H]
    \centering
    
\begin{center}
\begin{tikzpicture}[scale=2.0, baseline=0, x=2em, y=2em]
\hyperplanesBtwo

\fill[blue!20, opacity=0.45] (-1,-1)--(2,2)--(-1,2);

\draw[blue, very thick] (-1,-1)--(-1,2);
\draw[blue, very thick] (-1,-1)--(2,2);

\coordinate (P0) at ({-5/6},{-1/2});
\coordinate (P1) at ({-1/2},{-1/6});
\coordinate (P2) at ({-1/2},{1/6});
\coordinate (P3) at ({-1/6},{1/2});
\coordinate (P4) at ({-1/2},{5/6});
\coordinate (P5) at ({-1/2},{7/6});
\coordinate (P6) at ({-1/6},{3/2});
\coordinate (P7) at ({1/6},{3/2});
\coordinate (P8) at ({1/2},{11/6});

\coordinate (P) at ({-1},{-1});

\filldraw[black] (P) circle (1pt);  \node[anchor=east] at (P) {$y$};

\foreach \k in {0,...,8}{
    \filldraw[black] (P\k) circle (1pt);
    \node[anchor=west] at (P\k) {$x_{\k}$};}

\draw[red, very thick]
    (P0) -- (P1) -- (P2) -- (P3) --
    (P4) -- (P5) -- (P6) -- (P7) -- (P8);

\draw[red, dotted, very thick] 
    ({1/2},{11/6}) -- ({1/2},{7/3});

\begin{scope}[cm={1,0.5,0,0.866,(0,0)}]
        \node[rotate=45, anchor=west]  at (1,1) {$H_{\beta}$};
        \node[rotate=-45, anchor=east] at (-3/2,3/2) {$H_{-1-(\alpha+\beta)}$};
        \node[anchor=west]  at (-1, 3) {$H_{\alpha}$};
\end{scope}

\end{tikzpicture}\caption{\small Example of such a sequence in the case of $B_2.$}
\end{center}
\end{figure}

We now define $f_0(\alpha)=f(\alpha)$ and for $i\geq 1$ we set $$f_i(\alpha):=
f(\alpha)+\bigl\langle\alpha,x_i-x_0\bigl\rangle.$$ Each $f_i$ is still a concave function since the pairing is a linear function. We denote $J_i=J_{f_i}$ and we let $(J_i^+)_{i\geq 0}$ and $(J^-_i)_{i\geq 0}$ be $$J_i^+:=J_i\cap U^+,\hspace{0.5cm}J_i^-:=J_i\cap U^-.$$ We extend $\chi$ to a representation $\chi_i$ on $J_i$ that is trivial on $J_i^-$ and $J_i^+$. In fact, as for $J_f,$ we have 
$$T\cap\bigl\langle J_i^+, J_i^-\bigl\rangle\subset \ker(\chi).$$ In this way, Conditions 1,2 and 3 of Theorem \ref{blondel} are trivially satisfied.

Before proving the theorem, we will prove the following lemma. We now fix, for the rest of this paper, a Haar measure $\mu_U$ on $U$ which decomposes as $\mu_U= \mu_{\alpha_1}\cdots \mu_{\alpha_n}$ with $\mu_{\alpha_i}$ a Haar measure on $U_{\alpha_i}.$ 

\begin{lemma}\label{Lemma}
    Assume the above setting. Let $(\pi, V)$ be any representation of $G$, and let $v$ be a non-zero element in the $\chi_i$-isotypic component $V^{\chi_i}.$ Let $\alpha\in R(G,S)$ be the only indivisible positive root such that the root subgroup $U_{\alpha}$  is modified at the $i+1$-step. Then $$w:=\int_{U_{\alpha,f_{i+1}}}\pi(u_\alpha) v \hspace{0.2cm}d\mu_{\alpha}(u_{\alpha})$$ lies in $ V^{\chi_{i+1}}.$
\end{lemma}

\begin{proof}
    First of all, we notice that $w$ obviously lies in $V^{\chi},$ since $T^0$ normalizes $U_{\alpha,r}$ and $U_{2\alpha,r}$ for every $r>0.$
    We first look at the action of $J_{i+1}^+$ on $w.$ Let $u_{\beta}\in U_{\beta,f_{i+1}}$ be a non-trivial element for some positive root $\beta$. Clearly, if $\beta=\alpha$ or $\beta=2\alpha,$ then $$\pi(u_{\beta})\cdot w=w.$$
    So we can assume $\beta\not\in\{\alpha,2\alpha\}$ and $U_{\beta,f_{i+1}}=U_{\beta,f_i}.$ The element $u_\beta u_\alpha u_{\beta}^{-1}$ lies in 
    $$u_{\alpha}\prod_{p>0,q>0}U_{p\alpha+q\beta,pf_{i+1}(\alpha)+qf_{i+1}(\beta)}.$$
    Here we are using that in the product $\prod_{p>0,q\geq0}U_{p\alpha+q\beta,pf_{i+1}(\alpha)+qf_{i+1}(\beta)}$ the order of the terms can be chosen arbitrarily. Notice that since $p,q>0,$ and since $f_{i+1}$ is concave, we have $$pf_{i+1}(\alpha)+qf_{i+1}(\beta)\geq f_{i+1}(p\alpha+q\beta)=f_{i}(p\alpha+q\beta).$$ Finally, since the action of $u_{\beta}$ is trivial on $v$ we can write 

    \begin{align*}
    \pi(u_\beta)\cdot w&= \int_{U_{\alpha,f_{i+1}}}\pi(u_\beta u_\alpha) v \hspace{0.2cm}d\mu_{\alpha}(u_{\alpha})=\int_{U_{\alpha,f_{i+1}}}\pi(u_\beta u_\alpha u_{\beta}^{-1}) v \hspace{0.2cm}d\mu_{\alpha}(u_\alpha)\\ =& \int_{U_{\alpha,f_{i+1}}}\pi(u_\alpha) v \hspace{0.2cm}d\mu_{\alpha}(u_{\alpha}).
    \end{align*}

    Now we look at the action of $J_{i+1}^-.$ Let now $\beta$ be a negative root and $u_{\beta}\in U_{\beta,f_{i+1}}$ a non-trivial element. First, we assume $\beta\neq -\alpha.$ Then using a similar argument as above we can conclude that $\pi(u_{\beta})\cdot w=w.$ So, we only need to study the case $\beta=-\alpha.$ Then we can use \cite[Proposition 7.1.1 (1)]{KalPras} and its proof to write $u_{\beta}u_{\alpha}=u'_{\alpha}t u'_{\beta}$ with $u_{\beta}'\in U_{\beta,{f_{i+1}}}, u_{\alpha}'\in U_{\alpha,{f_{i+1}}}$ and $t\in T_{c_{\alpha}}\cap \alpha^{\vee}(F_{\alpha}^{\times}).$

    Since $J_{i+1}^-\subset J_i^-,$ the action of $u_{\beta}'$ is trivial on $v$. The action of $t$ is trivial as well. So, the action of $u_{\beta}$ is trivial and $w\in V^{\chi_{i+1}}.$   
\end{proof}

Now we can start with the proof of Theorem \ref{type}. We divide the proof in two steps:

\paragraph{Step 1: Rank one groups}

We first look at the case where $G$ has $F$-rank one. We already know that for $c_{\alpha}=0,$ $\chi$ is depth-zero and $(J_f,\chi_0)$ is a type for $[T,\chi]_G$ (see \cite{RochePrincipal} for the split case and \cite{morris1999level} for the quasi-split case). So we assume $c_{\alpha}>0.$ We want to check Condition 4 of Theorem \ref{blondel}. 

First, we study the case of $\SL_2$ and $\SU_3.$ In $\SL_2$ we can write the product of a matrix in $U_{\alpha}$ and a matrix in $U_{-\alpha}$ in the following way: $$\begin{pmatrix}
    1 & a \\ &1
\end{pmatrix}\cdot \begin{pmatrix}
   1 &  \\ a'&1
\end{pmatrix}=\begin{pmatrix}
    1 &  \\ \gamma^{-1}a'&1
\end{pmatrix}\begin{pmatrix}
    \gamma &  \\&\gamma^{-1}
\end{pmatrix}\begin{pmatrix}
    1 &\gamma^{-1}a\\&1
\end{pmatrix}$$ with $\gamma= 1+aa'.$

On the other hand, in $\SU_3$ we have the decomposition $$\begin{pmatrix}
    1 & a & b\\ &1& \overline{a}\\&&1
\end{pmatrix}\cdot \begin{pmatrix}
   1 &  & \\ a'&1& \\b'&\overline{a'}&1
\end{pmatrix}=\begin{pmatrix}
    1 &  & \\ a_1'&1& \\ b_1'&\overline{a_1'}&1
\end{pmatrix}\begin{pmatrix}
    \gamma &  & \\ &\gamma\overline{\gamma}^{-1}& \\&&\overline{\gamma}^{-1}
\end{pmatrix}\begin{pmatrix}
    1 & a_1 & b_1\\ &1& \overline{a_1}\\&&1
\end{pmatrix}$$ for $\gamma= 1+aa'+bb'$, $a_1 =( a+a'b)\gamma^{-1}$, $b_1=b\gamma^{-1}$, $a'_1=(a'+ab')\gamma^{-1}$,  $b'_1=b'\gamma^{-1}$.

From these, we understand the $T$-components of the commutators $[U_{\alpha,f_{i+1}},U_{-\alpha,f_i}].$ Picking $H_{i,x}= J_i^-$ is enough to ensure that $x H_{i,x}x^{-1}\subset J_i$ for every $x\in U_{\alpha,f_{i+1}}\setminus U_{\alpha,{f_i}} $, and that $\chi_i$ has no non-zero fixed vector, since $c_\alpha+\bigl\langle x_i-x_0\bigl\rangle<c_\alpha$. Therefore, Condition 4 of Theorem \ref{blondel} is satisfied, and by Proposition \ref{Blondel2}, $J_i$ is a type for $[T,\chi]_G$ for every $i.$ 

Assume now that $G$ is simply connected. Then $G$ is the Weyl restriction of scalars of either $\SL_2$ or $\SU_3.$ In this case, we can again choose $H_{i,x}=J_i^-=U_{-\alpha,f_i}$ and the computations for $\SL_2$ and $\SU_3$ can be generalized to $G$.

Finally, let $G$ be any rank one group, and let $G^{\Sc}$ denote the simply connected cover of the derived subgroup of $G$.  
On $G^{\Sc}$ we consider the subgroup
$$H_{i,x}^{\Sc} = U^{\Sc}_{-\alpha,f_i} \subset G^{\Sc}.$$
 
Since $H_{i,x}^{\Sc}$ is schematic, it descends along the central isogeny $G^{\Sc} \to G$ to the corresponding subgroup
$$H_{i,x}:= U_{-\alpha,f_i} \subset G,$$ preserving Condition 4 of Theorem \ref{blondel}.

This proves the theorem for groups with $F$-rank one. Notice that the same proof can be adapted to prove the statement for groups with $F$-rank one extended with some torus (as we will use in the next step).

\paragraph{Step 2: General quasi-split groups}

If $G$ is any quasi-split group, we notice that Blondel's theorem can't be applied if $c_{\alpha}=0$ for some $\alpha,$ since we can't expect to find a subgroup like $H_{i,x}$ if at the $i+1$-th step we modified the root subgroup associated to $\alpha$. 

Let $(\pi,V)$ be any representation of $G.$ Looking at the proof of \cite[Theoreme 1]{Blondel1}, Blondel uses Condition 4 only to prove that for every $v\in V^{\chi_0}\setminus \{0\}$ and for every $i$ 
\begin{equation}\label{implication}
    \int_{J_i^+}\pi(u)v\hspace{0.2cm} d\mu_U(u)\neq 0\text{ implies }\int_{J_{i+1}^+}\pi(u)v\hspace{0.2cm} d\mu_U(u)\neq 0.
\end{equation} 
This is \cite[Proposition 1]{Blondel1}. We will now prove that this is indeed true for our sequence $\{J_i\}_i$ even when Condition 4 of Theorem \ref{blondel} is not verified. We proceed by induction:

\begin{itemize}
    \item At step 0, clearly $\int_{J_0^+}\pi(u)v\hspace{0.2cm}d\mu_U(u)\neq 0$ if $v\in V^{\chi_0}\setminus\{0\}.$ 

    \item If Condition 4 is satisfied at the $i+1$-th step, if $$\int_{J_i^+}\pi(u)v\hspace{0.2cm} d\mu_U(u)\neq 0$$ then $$\int_{J_{i+1}^+}\pi(u)v\hspace{0.2cm} d\mu_U(u)\neq 0$$ for every $v\in V^{\chi_0}\setminus\{0\}$ \cite[Proposition 1]{Blondel1}. 

    \item Now we analyze the case in which, at the $i+1$-th step, only the root subgroups $U_{\alpha}$ and $U_{-\alpha}$ are modified, and $c_{\alpha}=0.$ We assume that the integral $\int_{J_i^+}\pi(u)v\hspace{0.2cm} d\mu_U(u)$ is not 0 for every $v\in V^{\chi_0},$ and without loss of generality, we set $\alpha=\alpha_1$. Then
    \begin{align*}
    w &= \int_{J_{i+1}^+}\pi(u)v\hspace{0.2cm} d\mu_U(u)\\ =& \int_{U_{\alpha,f_{i+1}}}\pi(u_1)\int_{U_{\alpha_2,f_{i}}}\cdots\int_{U_{\alpha_n,f_{i}}}\pi(u_2\cdots u_n)v\hspace{0.2cm}d\mu_{\alpha_n}(u_n) \cdots d\mu_{\alpha_1}(u_1)
    \end{align*}
    
    Since $\int_{J_i^+}\pi(u)v\hspace{0.2cm} d\mu_U(u)\neq 0$, we have $$\int_{U_{\alpha_2,f_{i}}}\cdots\int_{U_{\alpha_n,f_{i}}}\pi(u_2\cdots u_n)v\hspace{0.2cm} d\mu_{\alpha_n}(u_n)\cdots d\mu_{\alpha_2}(u_2)\neq 0.$$ 

    Suppose that at the first step we modified the root subgroups associated with a certain root $\alpha_k.$ Then we can reorder the integrals moving the $\alpha_k$ integral to the right, getting $$w=\int_{U_{\alpha,f_{i+1}}}\pi(u_1)\int_{U_{\alpha_2,f_{i}}}\cdots\int_{U_{\alpha_k,f_{i}}}\pi(u_2\cdots u_k)v\hspace{0.2cm}d\mu_{\alpha_k}(u_k) \cdots d\mu_{\alpha_1}(u_1).$$

    Then $w$ is equal to $$C_1\int_{U_{\alpha,f_{i+1}}}\pi(u_1)\int_{U_{\alpha_2,f_{i}}}\cdots\int_{U_{\alpha_k,f_{i}}}\int_{U_{\alpha_k,f_1}}\pi(u_2\cdots u_k u_k')v\hspace{0.2cm} d\mu_{\alpha_k}(u_k')d\mu_{\alpha_k}(u_k)\cdots d\mu_{\alpha_1}(u_1),$$ for some constant $C_1.$ Using now Lemma \ref{Lemma}, we can assume $v\in V^{\chi_1}.$ Repeating this argument $i$ times, we can assume $v\in V^{\chi_i}.$

    With this assumption, we get 
    \begin{align*}
    w &= \int_{U_{\alpha,f_{i+1}}}\pi(u_1)\int_{U_{\alpha_2,f_{i}}}\cdots\int_{U_{\alpha_n,f_{i}}}\pi(u_2\cdots u_n)v\hspace{0.2cm} d\mu_{\alpha_n}(u_n)\cdots d\mu_{\alpha_1}(u_1)\\ =& C\int_{U_{\alpha,f_{i+1}}}\pi(u_1)v\hspace{0.2cm}d\mu_{\alpha_1}(u_1),
    \end{align*}
     for some constant $C.$ 

    We then look at the group $G_{\alpha}:= \bigl\langle U_{\alpha},T,U_{-\alpha}\bigl\rangle$ which consists of a group with $F$-rank one extended with some torus. From Step 1, the subgroup $J_{i}^{\alpha}=\bigl\langle U_{\alpha,f_i},T^0,U_{-\alpha,f_i}\bigl\rangle$ and $\chi_i|_{J_i^{\alpha}}$ form a type for $[T,\chi]_{G_{\alpha}}$. Therefore
    $$\int_{U_{\alpha,f_{i+1}}}\pi|_{G_{\alpha}}(u_1)v'\ d\mu_{\alpha_1}(u_1)\neq 0\text{ for every }v'\in V^{\chi_i|_{J_i^{\alpha}}}\setminus\{0\} ,\text{ for every }k\geq 0.$$ 

    Since $v$ is in $V^{\chi_i}$ and it is not 0 since $\int_{J_i^+}\pi(u)v\hspace{0.1cm}d\mu_U(u)\neq 0,$ we conclude that $\int_{J_{i+1}^+}\pi(u)v\hspace{0.1cm}d\mu_U(u)\neq 0.$

    \item What is left to show is that if $c_\alpha\neq 0,$ then at every step in which we modify $U_{\alpha}$ Condition 4 of Theorem \ref{blondel} is satisfied. So we assume at the $i+1$-th step we modify $U_{\alpha}=U_{\alpha_1}$ and $c_{\alpha}\neq 0$. Let $x\in J^+_{f_{i+1}}\setminus J^+_{f_i},$ and let $H_{i,x}:= U_{-\alpha, f_i}\subset J_{f_i}^-.$ We write $x=u_{\alpha_2}\cdots u_{\alpha_n}u_{\alpha},$ with $u_{\alpha_j}\in U_{\alpha_j,f_i}$ and $u_{\alpha}\in U_{\alpha,f_{i+1}}.$ We analyze $x H_{i,x}x^{-1}.$ First, by the case of rank 1 subgroups, we know that $$u_{\alpha} H_{i,x} u_{\alpha}^{-1}  \subset J_{i}.$$ Since $u_{\alpha_i}\in J_i,$ we get $$x H_{i,x}x^{-1}\subset J_i.$$ 

    The only torus component comes from $u_{\alpha} H_{i,x} u_{\alpha}^{-1}$, and we have already noticed in the case of rank 1 groups that it is not contained in the kernel of $\chi,$ meaning that there are no non-zero fixed vectors in $\chi_i.$ So Condition 4 of Theorem \ref{blondel} is satisfied. 
    \end{itemize}
    
    Now, we have shown that \ref{implication} holds, so that $$\int_{J_i^+}\pi(u)v\hspace{0.2cm} d\mu_U(u)\neq 0\implies\int_{J_{i+1}^+}\pi(u)v\hspace{0.2cm} d\mu_U(u)\neq 0$$ for every $v\in V^{\chi_0}\setminus\{0\}$ and for every $i$, and using the same argument as in the \cite[Theoreme 1]{Blondel1} we conclude $(J_f,\chi_0)$ is a type for $[T,\chi]_G$.\qed

%% file: Parameters.tex
\section{Hecke algebra for L-parameters of principal series}\label{HeckeLpar}

After fixing a separable closure $F_s$ of $F,$ we denote by $I_F\subset W_F\subset \Gal(F_s/F)$ the inertia subgroup and the Weil group, respectively. We fix, once and for the remainder of this paper, a geometric Frobenius element $\Fr$ in $W_F.$ We recall that a $L$-parameter for $G$ is a continuous homomorphism $$\varphi:W_F\times\SL_2(\mathbb{C})\to G^{\vee}\rtimes W_F=:{}^LG$$ with some extra properties:\begin{itemize}
    \item $\varphi|_{\SL_2(\mathbb{C})}$ is algebraic;
    \item $\varphi(W_F)$ consists of semisimple elements.
    \item $\varphi(w)\in G^{\vee}w$ for every $w\in W_F.$

\end{itemize}
We say that $\varphi$ is discrete if it does not factor through the $L$-group of any proper Levi subgroup of $G$. Equivalently, $\varphi$ is discrete if its centralizer in $G^{\vee}$ is finite modulo $Z(G^{\vee})^{W_F},$ where $Z(G^{\vee})$ denotes the center of $G^{\vee}.$ We say that $\varphi$ is tempered if its image projects onto a bounded subset in $G^{\vee}.$ An irreducible representation $\rho\in \Irr(\pi_0(Z_{G^{\vee}}(\varphi)))$ of the component group of the centralizer of $\varphi$ in $G^{\vee}$ is called an enhancement of $\varphi.$ We denote the set $G^{\vee}$-conjugacy classes of enhanced $L$-parameters of $G^{\vee}$ by $\Phi_e(G).$ Recall that conjecturally, the set of discrete series representations of $G$ should be in bijection with the set of $G^{\vee}$-conjugacy classes of tempered discrete $L$-parameters via the local Langlands correspondence.

We will now focus on $L$-parameters coming from principal series representations of quasi-split groups. Most of the things we will discuss come from \cite{AffineHeckeForLparameters}, but the easier formulation for principal series comes from \cite{quasisplitps}. 

Since the dual group $T^{\vee}$ of the maximal torus $T$ is a complex torus, any $L$-parameter for $T$ is trivial on $\SL_2(\mathbb{C})$ and corresponds to the $T^{\vee}$-conjugacy class of a homomorphism $$\varphi: W_F\to {}^LT.$$ 

Parameters coming from principal series can be described in terms of cohomology groups of flag varieties. In particular, let 
$$H^{\vee}:= Z_{G^{\vee}}(\varphi(W_F))$$
and let $u_{\varphi}=\varphi(1,\begin{psmallmatrix} 1&1\\0&1    
\end{psmallmatrix})$. Then an enhanced $L$-parameter $(\varphi,\rho)\in \Phi_e(G)$ is a principal series $L$-parameter if after replacing it by suitable $G^{\vee}$-conjugate:

\begin{itemize}
    \item $\varphi(\Frob_F,\begin{psmallmatrix}q_F^{-1/2}&0\\ 0&q_F^{1/2}\end{psmallmatrix})$ 
    and $\varphi(x)$ belong to $T^{\vee}\rtimes W_F$ for every $x\in I_F.$
    \item Any irreducible constituent $\rho^{\circ}$ of $\rho|_{\pi_0(Z_{H^{\vee}}(u_{\varphi}))}$ appears in the homology 
    $H_*(\mathcal{B}_{H^{\vee}}^{u_{\varphi}})$ where $\mathcal{B}_{H^{\vee}}^{u_{\varphi}}$ is the variety of Borel subgroups of $H^{\vee,\circ}$ containing $u_\varphi.$
\end{itemize}

In \cite{GeneralizationSpringer} a cuspidal support map ${}^L\Psi$ for $L$-parameters was defined. Using this, in \cite[Definition 3.5]{AffineHeckeForLparameters} the definition of inertial equivalence class of enhanced $L$-parameters was given:

\begin{definition}\cite[Definition 3.5]{AffineHeckeForLparameters}
An inertial equivalence class for $\Phi_e(G)$ is the $G^{\vee}$-conjugacy class $\mathfrak{s}^{\vee}$ of a pair $(\mathcal{L}^{\vee}\rtimes W_F,\mathfrak{s}^{\vee}_{\mathcal{L}}),$ where $\mathcal{L}(F)$ is a Levi subgroup of $G$ and $\mathfrak{s}^\vee_L$ is a $X_{\nr}(G)$-orbit in the set of cuspidal $L$-parameters $\Phi_{\cusp}(L(F))$. This yields a Bernstein decomposition 

$$\Phi_e(G)=\bigsqcup_{\mathfrak{s}^{\vee}\in\mathfrak{Be}^{\vee}(G)}\Phi_e(G)^{\mathfrak{s}^{\vee}}=\bigsqcup_{\mathfrak{s}^{\vee}\in\mathfrak{Be}^{\vee}(G)}{}^L\Psi^{-1}(L^{\vee}\rtimes W_F,\mathfrak{s}_L^{\vee}),$$ 
where $\mathfrak{Be}^{\vee}(G)$ denotes the set of inertial equivalence classes of $\Phi_e(G).$ 
\end{definition}

To an inertial equivalence class $\mathfrak{s}^{\vee}$ for $\Phi_e(G),$ we associated the finite group $$W_{\mathfrak{s}^{\vee}}:=\text{stabilizer of }\mathfrak{s}^{\vee}_{\mathcal{L}}\text{ in }N_{G^{\vee}}(\mathcal{L}^{\vee}\rtimes W_F)/\mathcal{L}^{\vee}.$$ 

Recall that in Section \ref{TypeSection}, we fixed an inertial equivalence class $\mathfrak{s}=[T,\chi]_G.$ Let $\varphi_{\chi}$ be the $L$-parameter associated to $\chi$ Let $\mathfrak{s}_T$ and $\mathfrak{s}^{\vee}_T$ the Bernstein component of $\chi$ for $T$ and the Bernstein component of $\varphi_{\chi}$ for $T^{\vee}.$ Then $\mathfrak{s}_T^{\vee}$ gives rise to a principal series Bernstein component $\Phi_e(G)^{\mathfrak{s}^{\vee}}$. Let 
$$J:=Z_{G^{\vee}}(\varphi_{\chi}|_{I_F}).$$ 
Then the Weyl group 
$$W_{\mathfrak{s}^{\vee}}^{\circ}:=W(R(J^{\circ},(T^{\vee,W_F})^{\circ}))$$
is a normal subgroup of $W_{\mathfrak{s}^{\vee}},$ and $$W_{\mathfrak{s}^{\vee}}=W_{\mathfrak{s}^{\vee}}^{\circ}\rtimes \Gamma_{\mathfrak{s}^{\vee}}$$ where $\Gamma_{\mathfrak{s}^{\vee}}$ is defined as the stabilizer in $W_{\mathfrak{s}^{\vee}}$ of the set of positive roots determined by $B^{\vee}$ in $R(J^{\circ},(T^{\vee,W_F})^{\circ}).$ To $\mathfrak{s}^{\vee}$ we associate a root datum 
$$\mathcal{R}_{\mathfrak{s}^{\vee}}=\big(R_{\mathfrak{s}^{\vee}},X^*((T^{\vee,I_F})^{\circ}_{W_F},R_{\mathfrak{s}^{\vee}},X_*((T^{\vee,I_F})^\circ_{W_F})\big)$$
where \cite[Lemma 4.1]{quasisplitps}
$$R_{\mathfrak{s}^{\vee}}=\{f_{F_{\alpha/F}}\alpha^{\vee}:\alpha^{\vee}\in R(J^{\circ},(T^{\vee,I_F})^{\circ})_{\red}\}.$$ 
Let 
$$\mathcal{H}(\mathfrak{s}^{\vee},q_F^{1/2})^\circ=\mathcal{H}(\mathcal{R}_{\mathfrak{s}^{\vee}},\lambda,\lambda^*,q_F^{1/2})$$
be the affine Hecke algebra with parameter function determined in \cite[Proposition 3.14]{AffineHeckeForLparameters}, and consider the extended affine Hecke algebra
$$\mathcal{H}(\mathfrak{s}^{\vee},q_F^{1/2})=\mathcal{H}(\mathfrak{s}^{\vee},q_F^{1/2})^{\circ}\rtimes \Gamma_{\mathfrak{s}^{\vee}}$$ constructed in \cite{AffineHeckeForLparameters}. It was proven \cite[Theorem 3.18]{AffineHeckeForLparameters} that the set of irreducible representations of $\mathcal{H}(\mathfrak{s}^{\vee},q_F^{1/2})$ is canonically in bijection with $\Phi_e(G)^{\mathfrak{s}^{\vee}}.$ This result goes through the two-steps Lusztig's reduction process described in Section \ref{ReductionSection} and a classification of irreducible modules of geometric graded Hecke algebras done in \cite[Corollary 3.23]{GradDisc}.

In \cite{quasisplitps} a local Langlands correspondence for principal series of quasi-split groups was defined. The correspondence is obtained by identifying $\mathcal{H}(\mathfrak{s}^\vee, q_F^{1/2})$ with the algebra $\mathcal{H}(\mathfrak{s})^{\Op}\cong\End_G(\Pi_{\mathfrak{s}})^{\Op}$ where $\Pi_{\mathfrak{s}}$ is a progenerator for the category $\Rep(G)^{\mathfrak{s}}$. The isomorphism 
$$\mathcal{H}(\mathfrak{s}^{\vee},q_F^{1/2})\cong \mathcal{H}(\mathfrak{s})^{\Op}$$
comes from identifying first the root data and then the finite groups $\Gamma_{\mathfrak{s}^{\vee}}$ and $\Gamma_{\mathfrak{s}}$.

It is well known \cite[Corollary 7.12]{StructureTheoryViaTypes} that whenever there is a type, the algebra $\End_G(\Pi_{\mathfrak{s}})^{\Op}$ is Morita equivalent to the Hecke algebra of the type. In the case of principal series, this Morita equivalence actually comes from an isomorphism \cite[\S 1.5]{Dat}. Therefore, in the bijection from \cite[Theorem 7.1. and its proof]{quasisplitps} we can replace $\End_G(\Pi_{\mathfrak{s}})^{\Op}$ with $\mathcal{H}(J_f,\chi_0)^{\Op}$ obtaining 
$$\Irr(G)^{\mathfrak{s}}\leftrightarrow\Irr(\mathcal{H}(J_f,\chi_0)^{\Op})\leftrightarrow\Irr(\mathcal{H}(\mathfrak{s}))^{\Op}\leftrightarrow\Irr(\mathcal{H}(\mathfrak{s}^{\vee},q_F^{1/2}))\leftrightarrow\Phi_e(G)^{\mathfrak{s^{\vee}}}.$$

From now on, if $\pi\in \Irr(G)^{\mathfrak{s}},$ we will denote by $(\varphi_\pi,\rho_\pi)$ the enhanced $L$-parameter associated to it.

%% file: Gamma.tex
\section{Computations of Epsilon Factors}
\label{EpsilonSec}

Before starting to prove the formal degree conjecture, we need to study the adjoint $\gamma$-factors involved. To compute these factors, we would like to proceed as in the split case, and split them in two pieces, one coming from some unramified $L$-parameter (for which we can use \cite[Appendix 1]{UnipotentFormalDegree} to relate them to some formal degree), and one for which we can compute the absolute value of the $\varepsilon$-factor. 

We write $\mathfrak{g}^{\vee}:=\Lie(G^{\vee})$ and $\mathfrak{t}^{\vee}:=\Lie(T^{\vee}),$ let $Z(G)_s$ the split part of $Z(G)$ and write $\mathfrak{a}^{\vee}=\Lie(Z(G^{\vee})^{W_F})$ so that $\mathfrak{g}^{\vee}/\mathfrak{a}^{\vee}=\Lie((G/Z(G)_s)^{\vee}).$  We fix $\psi$ an order-0 additive character of $F$, and we recall from \cite{ArithmeticInv} the following meromorphic functions in $s\in\mathbb{C}$ attached to any $L$-parameter $\varphi$: $$L(s,\Ad_{G^{\vee}}\circ \varphi)=\det\left(1-q_F^{-s}(\Ad_{G^{\vee}}\circ\varphi(\Frob)|_{(\mathfrak{g}^{\vee}/\mathfrak{a}^{\vee})^{\varphi(I_K)}})\right)^{-1},$$
$$\gamma(s,\Ad_{G^{\vee}}\circ\varphi,\psi)=\varepsilon(s,\Ad_{G^{\vee}}\circ\varphi,\psi)\frac{L(1-s,\Ad_{G^{\vee}}\circ \varphi)}{L(s,\Ad_{G^{\vee}}\circ\varphi)},$$
where $\varepsilon$ is the local factor defined in \cite{ArithmeticInv} as a slight modification of \cite[4.1.6]{Tate}.

For every relative root $\alpha\in R(\mathcal{G},\mathcal{S})$ we choose $\tilde{\alpha}\in R(\mathcal{G},\mathcal{T})$ preimage of $\alpha.$ Notice that every $\tilde{\alpha}\in R(\mathcal{G},\mathcal{T})$ restricts to a non-trivial character of $S$ since $G$ is quasi-split. We write 
$$\mathfrak{g}^{\vee}_{\alpha^{\vee}}:=\bigoplus_{\sigma\in \Gal(F_\alpha/F)}\Lie(U^{\vee}_{{}^\sigma\tilde{\alpha}^{\vee}})=\bigoplus_{\sigma\in \Gal(F_\alpha/F)}\mathfrak{g}^{\vee}_{{}^\sigma\tilde{\alpha}^{\vee}}.$$ 
If $F_\alpha=F,$ then the action of $\varphi(I_F)$ on $\mathfrak{g}^{\vee}_{\alpha^{\vee}}$ is trivial if and only if $c_\alpha=0.$ On the other hand, if $F_\alpha\neq F,$ the action of $\varphi(I_F)$ permutes the summand of $\mathfrak{g}^{\vee}_{\alpha^{\vee}}$, so it is never trivial but it might have some fixed points. This means that, differently from the split case, the $L$-functions might not be trivial on $\mathfrak{g}^{\vee}_{\alpha^{\vee}}.$

We split the $\gamma$-factor

$$\gamma(s,\Ad_{G^{\vee}}\circ\varphi,\psi)=\gamma(s,\Ad_{G^{\vee}}\circ\varphi|_{\mathfrak{t}^{\vee}/\mathfrak{a}^{\vee}},\psi)\cdot\prod_{\alpha\in R(\mathcal{G},\mathcal{S})}\gamma(s,\Ad_{G^{\vee}}\circ\varphi|_{\mathfrak{g}_{\alpha^{\vee}}^{\vee}},\psi),$$

and we now focus on studying the value of the $\varepsilon$-factor on $\mathfrak{g}_{\alpha^{\vee}}^{\vee}.$ Notice that if $c_\alpha=0,$ then the action of the inertia is trivial and therefore the $\varepsilon$-factor is trivial. So, we can assume $c_\alpha\neq 0.$ The absolute value of the $\varepsilon$-factors was studied in \cite{ArithmeticInv}. They proved that if $(\tau,W)$ is any Weil-Deligne representation, we have 
$$\abs{\varepsilon(0,\tau,\psi)}=q_F^{a(W)/2},$$
where $a(W)\in\mathbb{Z}_{\geq 0}$ is the Artin conductor of $W$ and is defined as follows. The inertial image $D_0=\tau(I_F)$ is the Galois group of a finite extension $F'/F$. If $A'$ is the ring of integers of $F'$ with maximal ideal $P'$, then we define the normal subgroup $D_j\leq D_0$ as the kernel of the action of $D_0$ on $A'/(P')^{j+1}$. This gives the lower indexing ramification filtration $$D_0\geq D_1\geq D_2\geq\cdots\geq D_n=1.$$
We write $d_j=|D_j|$. The Artin conductor is defined by $$a(W) = \sum_{j\geq 0}\dim(W/W^{D_j})\frac{d_j}{d_0}.$$

The $\varepsilon$-factor defined by Tate is inductive in degree 0 \cite[Theorem 3.4.1]{Tate}. Using this, Feng, Opdam and Solleveld proved that the $\varepsilon$-factor we are using is inductive in degree 0 as well \cite[Theorem A.2]{FengOpdamSolleveld}. This means that if we have extensions $F''/F'/F,$ then 
\[\begin{tikzcd}
	{R^0(W_{F'})} && {R^0(W_{F''})} \\
	& {\mathbb{C}^{\times}}
	\arrow["\Ind_{W_{F'}}^{W_{F''}}", from=1-1, to=1-3]
	\arrow["\varepsilon"', from=1-1, to=2-2]
	\arrow["\varepsilon", from=1-3, to=2-2]
\end{tikzcd}\] where $R^0$ denotes the set of virtual representations of formal dimension 0.

The stabilizer $I_{F_{\alpha}}$ of $\mathfrak{g}^{\vee}_{\tilde{\alpha}^{\vee}}$ in $I_F,$ is a subgroup of $I_F$ and $\mathfrak{g}_{\tilde{\alpha}^{\vee}}^{\vee}$ is a representation of $I_{F_{\alpha}}$. Since
$$\mathfrak{g}_{\alpha^{\vee}}^{\vee}\cong\ind_{I_{F_{\alpha}}}^{I_F}\mathfrak{g}_{\tilde{\alpha}^{\vee}}^{\vee}$$
as $I_F$-representation, we can compute the value of the $\varepsilon$-factor on $\mathfrak{g}_{\alpha^{\vee}}^{\vee}$ using its inductivity on degree 0. Since we have inductivity only in degree 0, we consider the virtual representation $\mathfrak{g}^{\vee}_{\tilde{\alpha}^{\vee}}-\triv_{I_{F_{\alpha}}}.$ Then, by inductivity in degree 0, we have $$\varepsilon(0,\mathfrak{g}^{\vee}_{\tilde{\alpha}^{\vee}}-\triv_{I_{F_{\alpha}}},\psi)=\varepsilon(0,\mathfrak{g}_{\alpha^{\vee}}^{\vee}-\ind_{I_{F_{\alpha}}}^{I_F}\triv_{I_{F_{\alpha}}},\psi).$$ The Artin conductor of the trivial representation is just 1, so we only need to compute the Artin conductor of the induction of the trivial representation and of $\mathfrak{g}_{\tilde{\alpha}^{\vee}}^{\vee}.$

Recall that from \cite[Proposition VI.2.5 + Corollary VI.2.1’]{LocalFields} we have a relation between Artin conductors and Herbrand functions. In particular, it says that if $E/F$ is any extension, and $\sigma$ is a character of a degree-1 representation $\tau$ of $\Gal(E/F)$, then 
$$a(\tau)=1+\varphi_{L/K}(r)$$
where $r$ is the smallest integer for which $\tau(\Gal(E/F)_r)\neq 0,$ $\Gal(E/F)_r$ is the $r$-th member of the lower indexing ramification filtration of $\Gal(E/F).$ Applying this to $\mathfrak{g}_{\tilde{\alpha}^{\vee}}^{\vee}$ as an $I_{F_{\alpha}}$-representation, we get 
$$a(\mathfrak{g}_{\tilde{\alpha}^{\vee}}^{\vee})=1+c^{\alpha}$$
with $c^{\alpha}$ the smallest integer such that $(I_{F_{\alpha}})^{c^{\alpha}}$ acts trivially on $\mathfrak{g}_{\tilde{\alpha}^{\vee}}^{\vee}.$ 

It was proven in \cite[Proposition 1.2]{OnDepthZeroChar} that restriction from $W_F\to W_{F_{\alpha}}$ on the $L$-parameter side, corresponds to the norm map $T(F_{\alpha})\to T(F)$ on the automorphic side, so 
$$c^\alpha = c_\alpha$$
where $c_\alpha$ is the one used in Section \ref{TypeSection} for constructing the type. 

The Artin conductor for the induction of the trivial representation was studied in \cite[Theorem A.2]{FengOpdamSolleveld}. We have $$a(\ind_{I_{F_{\alpha}}}^{I_F}\triv_{I_{F_{\alpha}}})=\val_L(\mathcal{D}_{F_{\alpha}/F})f_{F_{\alpha}/K},$$ where $\mathcal{D}_{F_{\alpha}/F}$ is the determinant of $F_{\alpha}/F$ and $f_{F_{\alpha}/F}$ its residual degree. 

We get
$$
|\varepsilon(0,\mathfrak{g}_{\alpha^{\vee}}^{\vee},\psi)|=|\varepsilon(0,\mathfrak{g}_{\tilde{\alpha}^{\vee}}^{\vee},\psi)||\varepsilon(0,\ind_{I_{F_{\alpha}}}^{I_F}\triv,\psi)| =q_F^{f_{F_{\alpha}/F}(1+c^{\alpha})+\val_{F_{\alpha}}(\mathcal{D}_{F_{\alpha}/F})f_{F_{\alpha}/F}}.
$$
Putting it all together, we obtain:

\begin{lemma}\label{EpsilonLemma}

With notation as above, consider $\mathfrak{g}^{\vee}$ as an $I_F$-representation. Then

$$|\varepsilon(0,\mathfrak{g}^{\vee},\psi)|=q_F^{a(\mathfrak{t}^{\vee})/2}\prod_{\alpha\in R(\mathcal{G},\mathcal{S})^+, c_\alpha\neq 0}q_F^{f_{F_\alpha/F}(1+c_{\alpha})+f_{F_\alpha/F}\val_{F_{\alpha}}{(\mathcal{D}_{F_{\alpha}/F})}}.$$
\end{lemma}

Notice that if $F_{\alpha}/F$ is unramified, then the valuation of the different is 0. For ramified groups (more generally, non-unramified groups) it was already noted by Solleveld \cite[\S 2, equation (18)]{RamifiedUnipotent} that the Haar measure on $G$ should be normalized in such a way so that the Artin conductor of $\mathfrak{t}^\vee$ and the valuation of the various determinants cancel out. We will explain this in Section \ref{FD}.

%% file: FD.tex
\section{Formal degree conjecture for principal series}\label{FD}

In this section, we will prove Conjecture \ref{HII} for principal series of quasi-split groups. We will just need to glue the pieces that we have built up so far, in a similar way as in \cite{Ricci}. We start by recalling the definition of the formal degree and fixing some notation for the remainder of this paper.

For a closed subgroup $A$ of the center $Z(G)$ of $G$ such that $Z(G)/A$ is compact, let $\mu_{G/A}$ be a Haar measure on $G/A$. For a smooth irreducible representation $(\pi,V)$ of $G,$ the matrix coefficient $\pi_{v,w}$ of $\pi$ with respect to $v\in V$ and $w$ in the smooth dual $V^*$ of $V$, is the function $\pi_{v,w}: G\to\mathbb{C}$ defined by $$\pi_{v,w}(x)=\bigl\langle\pi(x)v,w\bigl\rangle.$$ 

If the central character of $\pi$ is unitary, the function $g\mapsto\abs{\pi_{v,w}(g)}^2$ is constant on cosets of $A$, and hence defines a function on $G/A$. We say that $(\pi, V)$ is a discrete series if it has a unitary central character and 
$$\int_{G/A}\abs{\pi_{v,w}(g)}^2d\mu_{G/A}(g)<\infty\text{ for all }v\in V,w\in V^*.$$
Notice that this condition is independent of $A$, and also of the choice of a Haar measure on $G/A$. We already mentioned in Section \ref{EpsilonSec} that discrete series representations conjecturally correspond to tempered discrete $L$-parameters under the local Langlands correspondence. Since every discrete series is unitary there is a canonical isomorphism between $V$ and $V^*$, and we can define a matrix coefficient $\pi_{v,w}$ for $v,w\in V$. We write $\pi_v=\pi_{v,v}.$

\begin{definition}
    Let $(\pi, V)$ be a discrete series representation of $G$. As in \cite[\S 1]{HII}, we fix $A=Z(G)_s$ the maximal $F$-split central torus in $G$. Let $\mu_{G/A}$ be a Haar measure on $G/A$. Then, the formal degree $\fdeg(\pi)=\fdeg(\pi,\mu_{G/A})$ of $\pi$ is the unique number (depending on $\mu_{G/A}$) such that for all $v\in V,$$$\int_ {G/A}\abs{\pi_v(g)}^2\hspace{0.1cm}d\mu_{G/A}(g)=\frac{\abs{v}^2}{\fdeg(\pi)}.$$
\end{definition}

Let $\psi$ be any additive character of $F$ and denote by $\omega_G$ the canonical Haar measure from \cite[\S 5]{HaarMeasureArtinCond}. Then we define $$\mu_{G,\psi}=q_F^{-(a(\mathfrak{g}^{\vee})+\ord(\psi)\dim(G/Z(G)_s))/2}\abs{\omega_G},$$ where $a(\mathfrak{g}^\vee)$ is the Artin conductor of the action of just $I_F$ on $\mathfrak{g}^{\vee}$ (different from before, where we consider the action of $T^{\vee}\rtimes I_F$). This is the Haar measure considered by Solleveld in \cite{RamifiedUnipotent} to prove the Conjecture \ref{HII} for unipotent representations of ramified groups. Notice that 
$$q_F^{a(\mathfrak{g}^{\vee})}=q_F^{a(\mathfrak{t}^{\vee})+\sum_{\alpha\in R(\mathcal{G},\mathcal{S})}f_{F_\alpha/F}\val(\mathcal{D}_{F_\alpha/F})}.$$

Since we fixed $\psi$ an order-0 character, the normalizing factor simplifies. From now on, every $p$-adic group is going to be equipped with this Haar measure.

\subsection{Reduction to unipotent representations and computations}

To compute the formal degree of a representation in $\Irr(G)^{\mathfrak{s}}$ in terms of the associated $L$-parameter, we want to reduce to the case of unipotent representations. In that context, Conjecture \ref{HII} was already proven in \cite{UnipotentFormalDegree}. 

Recall that we denoted $J=Z_{G^{\vee}}(\varphi_{\chi}|_{I_F}).$ This is a complex group that contains $(T^{\vee})^{I_F,\circ}$ as a maximal torus. We define an action of $W_F$ on $J^{\circ}$ in the following way: consider
$$w\cdot x=\varphi_\chi(w)x\varphi_\chi(w^{-1})\text{ for every }x\in J^{\circ}.$$ 

Notice that $w\cdot x=x$ whenever $w\in I_F$ and $x\in J^{\circ}.$ Since $B^{\vee}$ is $W_F$-stable in ${}^LG$ and $\im(\varphi_\chi)\subset B^{\vee}\rtimes W_F,$ this action stabilizes the maximal torus $(T^{\vee})^{I_F,\circ}$ and the set of positive roots defined by $J^{\circ}\cap B^{\vee}.$ Let $\Delta_J$ be the basis of $R(J^{\circ},(T^{\vee})^{I_F,\circ})$ induced by $J^{\circ}\cap B^{\vee}.$ We would like to define an $L$-group ${}^LJ^{\circ}=J^{\circ}\rtimes W_F$ via this action, but if we fix a pinning $\{x_{\alpha^{\vee}}\}_{\alpha^{\vee}\in \Delta_J}$ the action of the Weil group might not preserve it. The inertia subgroup obviously preserves it but the action of the Frobenius might rescale the elements in the pinning via 
$$\Fr\cdot x_\alpha=c_{\alpha}x_{\Fr\cdot \alpha}$$ 
for some $c_\alpha\in\mathbb{C}.$ Since $(T^{\vee})^{I_F,\circ}$ is a maximal torus for $J^{\circ},$ we can pick an element $t\in (T^{\vee})^{I_F,\circ}$ so that 
$$t\cdot c_{\alpha} x_{\Fr\cdot \alpha}\cdot t^{-1}=x_{\Fr\cdot\alpha}\text{ for every }\alpha\in \Delta_J,$$ 
and we consider the action of $W_F$ on $J^{\circ}$ given by the action on $B^{\vee}\cap J$ and $(T^{\vee})^{I_F,\circ}$ described before, and the action on $\{x_\alpha\}_{\alpha\in\Delta_J}$ given by conjugation with $t\varphi_\chi(\Fr).$ 

We form the $L$-group 
$${}^LJ^{\circ}:=J^{\circ}\rtimes W_F=J^{\circ}\rtimes\bigl\langle t\varphi(\Fr), \varphi(w)\text{ for }w\in I_F\bigl\rangle.$$ 

Since $t\in (T^{\vee})^{I_F,\circ},$ the group $\bigl\langle t\varphi(\Fr), \varphi(w)\text{ for }w\in I_F\bigl\rangle$ is actually isomorphic to $W_F$.

Let $\mathcal{J}$ a quasi-split $F$-group such that its Langlands dual is isomorphic to ${}^LJ^{\circ}.$ Notice that $\mathcal{J}$ splits over an unramified extension since the action of the inertia subgroup is trivial on it. 

The group ${}^LJ^{\circ}$ contains the group $(T^{\vee})^{I_F,\circ}\rtimes W_F.$ Moreover, we can see ${}^LJ^{\circ}$ as a subset of ${}^LG$. In doing so, the image of the Frobenius under $\varphi_\chi$ in ${}^LG,$ can be seen as $(t^{-1},t\varphi_\chi(\Fr))$ in ${}^LJ^{\circ}\subset G\rtimes W_F$.

Consider the $L$-parameter $\widetilde{\varphi_{\chi}}$ given by 
$$\widetilde{\varphi_{\chi}}(x)=\begin{cases}
(1,\varphi_\chi(x))  \text{ if } x\in I_F,\\[0.9ex]
\widetilde{\varphi_{\chi}}(\Fr)=(t^{-1},t\varphi_{\chi}(\Fr)),
\end{cases}$$ 
seen as an $L$-parameter for $\mathcal{J}(F).$ Let $\mathfrak{s}^{\vee}_J$ be the inertial equivalence class of $L$-parameters of $\mathcal{J}(F)$ associated with this $\widetilde{\varphi_\chi}.$ There is a bijection \cite{RamifiedUnipotent,UnipotentCorresp}
$$\Irr(\mathcal{H}(\mathfrak{s}_J^{\vee},q_F^{1/2}))\to \Phi_{e}(J^{\circ})^{\mathfrak{s}^{\vee}_J}.$$  

Notice that the root data associated with $\mathcal{H}(\mathfrak{s}_J^{\vee},q_F^{1/2})$ is 
$$\mathcal{R}_{\mathfrak{s}^{\vee}_J} =\big( R_{\mathfrak{s}^{\vee}_J},X^*((T^{\vee,I_F})^{\circ}_{W_F}),R_{\mathfrak{s}^{\vee}_J},X_*((T^{\vee,I_F})^\circ_{W_F})\big)$$
with 
$$R_{\mathfrak{s}^{\vee}_J}=\{f_{F_{\alpha/F}}\alpha^{\vee}:\alpha^{\vee}\in R(Z_{J^{\circ}}(\widetilde{\varphi_\chi}|_{I_F}),(T^{\vee,I_F})^{\circ})_{\red}\}.$$ 
Since $\widetilde{\varphi_\chi}$ is trivial on $I_F$, $Z_{J^{\circ}}(\widetilde{\varphi_\chi}|_{I_F})=J^{\circ},$ and since $f_{F_\alpha/F}$ doesn't change if we see $\alpha^{\vee}$ as a root of $G^{\vee}$ or $J^{\circ},$ we have 
$$\mathcal{R}_{\mathfrak{s}^{\vee}_J}=\mathcal{R}_{\mathfrak{s}^{\vee}},$$ 
with $\mathcal{R}_{\mathfrak{s}^{\vee}}$ the root data associated to $\mathfrak{s}^{\vee}.$ Therefore, the algebras $\mathcal{H}(\mathfrak{s}_J^{\vee},q_F^{1/2})^{\circ}$ and $\mathcal{H}(\mathfrak{s}^{\vee},q_F^{1/2})^{\circ}$ are isomorphic. This isomorphism is induced from an isomorphism of the commutative subalgebras 
$$\mathcal{O}(\mathfrak{s}_J^\vee)\to \mathcal{O}(\mathfrak{s}^\vee),$$ induced from the map $\mathfrak{s}_J^{\vee}\to\mathfrak{s}^{\vee}$ sending $\widetilde{\varphi_\chi}\to \varphi_\chi.$ Here we see $\mathfrak{s}^{\vee}$ and $\mathfrak{s}^{\vee}_J$ as algebraic varieties.

Moreover, the algebra $\mathcal{H}(\mathfrak{s}^{\vee}_J,q_F^{1/2})$ is actually not extended and just equal to $\mathcal{H}(\mathfrak{s}^{\vee}_J,q_F^{1/2})^{\circ}$ \cite[Theorem 4.4]{UnipotentCorresp}. We therefore have an induction map 
$$\mathcal{H}(\mathfrak{s}_J^{\vee},q_F^{1/2})\text{-}\Mod \xrightarrow[]{\Ind} (\mathcal{H}(\mathfrak{s}^{\vee},q_F^{1/2})^{\circ}\rtimes \Gamma_{\mathfrak{s}^{\vee}})\text{-}\Mod = \mathcal{H}(\mathfrak{s}^{\vee},q_F^{1/2})\text{-}\Mod.$$

On the other hand, on the $L$-parameter side, we can push forward any $L$-parameter $\widetilde{\theta}$ in $\Phi(J^{\circ})^{\mathfrak{s}^{\vee}_J}$ along the inclusion ${}^LJ^{\circ}\to{}^LG$ obtaining an $L$-parameter $\theta$ for in $\Phi(G).$ Notice that the image of $\widetilde{\theta}(\SL_2(\mathbb{C}))$ is inside $J^{\circ}\subset G^{\vee}.$ Moreover, $\theta|_{I_F}=\varphi_\chi|_{W_F}$ and $\widetilde{\theta}|_{W_F}=\widetilde{\varphi_\chi}|_{I_F}.$ Therefore, 

$$Z_{G^{\vee}}(\theta)=Z_J(\theta(\SL_2(\mathbb{C}),\varphi_\chi(\Fr))\supset Z_{J^{\circ}}(\widetilde{\theta}(\SL_2(\mathbb{C}),\varphi_\chi(\Fr)))=Z_{J^{\circ}}(\widetilde{\theta}),$$
and we have an induction map 
$$\Phi_e(J^{\circ})^{\mathfrak{s}^\vee_j}\xrightarrow[]{\Ind}\Phi_e(G)^{\mathfrak{\mathfrak{s}^\vee}},$$ 
obtained by sending an enhancement $\rho\in \Irr(\pi_0(Z_{J^{\circ}}(\theta)))$ to $\Ind_{\pi_0(Z_{J^{\circ}}(\theta))}^{\pi_0(Z_{G^{\vee}}(\widetilde{\theta}))}(\rho).$

Let $\pi\in \Rep(G)^\mathfrak{s}$ be an irreducible discrete series and let $(\varphi_\pi,\rho_\pi)$ be the enhanced $L$-parameter of $G$ associated to $\pi.$ We will denote by $\pi'$ the $\mathcal{H}(J_{\chi},\chi_0)$-module associated to $\pi.$ Then we have 
$$\fdeg(\pi)=\frac{\fdeg(\pi')}{\vol(J_\chi)},$$
where the factor $\vol(J_\chi)$ (with respect to our fixed Haar measure $\mu_{G,\psi}$) comes from the normalization needed to make the obvious embedding $\mathcal{H}(J_{\chi},\chi)\to \mathcal{H}(G)$ trace-preserving. We refer to \cite[4.2.2]{OpdamSpectralCorresp} for the definition of the formal degree of a module over a Hecke algebra, and to \cite[\S3,\S4]{TypesPlancherel}  for more information about the Plancherel measure of Hecke algebras of types. 

The isomorphism 
$$\mathcal{H}(J_\chi,\chi_0)^{\Op}\to \mathcal{H}(\mathfrak{s}^{\vee},q_F^{1/2})$$
is just given by identifying the root data and the finite groups extending the affine Hecke algebras, so it does not affect the formal degree.

By Clifford theory, we know that 
$$\pi'\cong \Ind_{\mathcal{H}(\mathfrak{s}^{\vee},q_F^{1/2})^{\circ}}^{\mathcal{H}(\mathfrak{s}_J^{\vee},q_F^{1/2})}(\pi_J\otimes\sigma)$$ 
where $\pi_J$ is some irreducible $\mathcal{H}(\mathfrak{s}_J^{\vee},q_F^{1/2})^{\circ}$-representation and $\sigma$ is a projective representation of $\Gamma_{\mathfrak{s}^{\vee},\pi_J}=\Stab_{\Gamma_{\mathfrak{s}^{\vee}}}(\pi_J)$ (actually a better study of extended affine Hecke algebras shows that $\sigma$ is an actual representation, but this is not going to be relevant for our computations).

With arguments similar to \cite[Lemma 4.5, 4.6]{Ricci}, we obtain 
\begin{equation}\label{firstFD}
    \fdeg(\pi')=\fdeg(\pi_J)\frac{\dim(\sigma)}{\abs{\Gamma_{\mathfrak{s}^{\vee},\pi_J}}}.
\end{equation}

\begin{lemma}\label{Lemma2}
    Let $I_{\mathcal{J}}$ the Iwahori subgroup of $\mathcal{J}(F)$ defined by the set of positive roots $R(J^{\circ},(T^{\vee})^{I_F,\circ})^+$. Then 
    
    $$\frac{\vol(I_{\mathcal{J}})}{\vol(J_\chi)}=\abs{\varepsilon(0,\Ad_{G^{\vee}}\circ\varphi_\chi,\psi)}.$$ 
\end{lemma}

\begin{proof}
     Let $I=\bigl\langle T^0,\{U_{\alpha,0}\}_{\alpha\in R(\mathcal{G},\mathcal{S})^+,} \{U_{\alpha,1}\}_{\alpha\in R(\mathcal{G},\mathcal{S})^-}\bigl\rangle$ be the Iwahori subgroup of $G$ associated to our fixed set of positive roots. Recall that we are normalizing the Haar measure on $G$ and $\mathcal{J}(F)$ so that 
     $$\vol(I)=q_F^{-a(\mathfrak{g}^{\vee})/2}\cdot\abs{\overline{I}}\cdot q_F^{-(\dim((G^{\vee})^{I_F})+\dim(\overline{I}))/2},$$
    
    $$\vol(I_{\mathcal{J}})= q_F^{-a(\Lie(J^{\circ}))/2} \cdot \abs{\overline{I_{\mathcal{J}}}} \cdot q_F^{-(\dim((J^{\circ})^{I_F})+\dim(\overline{I_{\mathcal{J}}}))/2},$$

    where $\overline{I}$ and $\overline{I_\mathcal{J}}$ consist of the $k_F$-points of the maximal reductive quotients of $I$ and $I_\mathcal{J}$ respectively.
    
    The inertia subgroup acts trivially on $J^{\circ}$ and on $\Lie(J^{\circ}),$ so $a(\Lie(J^{\circ}))=0.$ Since $(T^{\vee})^{I_F,\circ}\subset J,$ we can see the maximal unramified quotient subtorus of $T$ as a maximal torus of $\mathcal{J}(F).$ Therefore, from \cite[Corollary B.7.12]{KalPras} we have $$\abs{\overline{I}}=\abs{\overline{I_{\mathcal{J}}}},$$
    and    
    $$\frac{\vol(I_{\mathcal{J}})}{\vol(I)}=q_F^{a(\mathfrak{g})/2}\cdot q_F^{(\dim((G^{\vee})^{I_F}-\dim((J^{\circ}))/2}.$$

    Since the maximal torus in $(G^{\vee})^{I_F}$ is $(T^{\vee})^{I_F,\circ}$, the difference between $\dim(J^{\circ})$ and $\dim((G^{\vee})^{I_F})$ is just the difference of the number of absolute roots in these groups. 

    Let $\alpha\in R((G^{\vee})^{I_F},(T^{\vee})^{I_F,\circ}).$ Then, if $c_\alpha\neq 0,$ we have $\alpha^{\vee}\circ\varphi_\chi(I_F)\neq1$ and $\alpha^{\vee}$ is not in $R(J^{\circ},(T^{\vee})^{I_F,\circ}).$ On the other hand, if $c_\alpha=0,$ then $\alpha\in R(J^{\circ},(T^{\vee})^{I_F,\circ}).$ Therefore we have
   
    $$\frac{\vol(I_\mathcal{J})}{\vol(I)} = q_F^{a(\mathfrak{g}^\vee)/2}\cdot \prod_{\substack{\alpha\in R((G^{\vee})^{I_F},(T^{\vee})^{I_F,\circ}) ),\\ c_{\alpha}\neq0}} q_F^{1/2}= q_F^{a(\mathfrak{g}^\vee)/2}\cdot\prod_{\substack{\alpha\in R(\mathcal{G},\mathcal{S})^+,\\c_{\alpha}\neq0}} q_F^{f_{F_\alpha/F}},$$ 

    where the last equality comes from the fact that relative roots are $I_F$-invariant and each relative root $\alpha$ corresponds to $f_{F_\alpha/F}$ many absolute roots.

    On the other hand, one easily sees that
    $$\frac{\vol(I)}{\vol(J_\chi)}=\prod_{\alpha\in R(\mathcal{G},\mathcal{S})^+}q_F^{f_{F_{\alpha}/F}c_{\alpha}}.$$

    We computed the absolute value of the $\varepsilon$-factor in Section \ref{EpsilonSec}. Using Lemma \ref{EpsilonLemma} we get 

\begin{align*}
    \left| \varepsilon(0, \Ad_{G^{\vee}} \circ \varphi_{\chi}, \psi) \right| =& q_F^{a(\mathfrak{t}^{\vee})/2}\hspace{-0.5cm}\prod_{\substack{\alpha\in R(\mathcal{G},\mathcal{S})^+ \\ c_\alpha\neq 0}}q_F^{f_{F_\alpha/F}(1+c_{\alpha})+f_{F_\alpha/F}\val_{F_{\alpha}}{(\mathcal{D}_{F_{\alpha}/F})}}\\
 =& q_F^{a(\mathfrak{g}^{\vee})/2}\prod_{\substack{\alpha\in R(\mathcal{G},\mathcal{S})^+ \\ c_\alpha\neq 0}}q_F^{f_{F_\alpha/F}(1+c_{\alpha})}= \frac{\vol(I_{\mathcal{J}})}{\vol(J_\chi)}.\qedhere\end{align*}

\end{proof}

\begin{lemma}\label{Lemma1}
   We have
    
    $$\frac{\dim(\rho_{\pi_J})}{\abs{S_{\varphi_{\pi_J}}^{\sharp}}}= \frac{\dim(\rho_{\pi})}{\abs{S^{\sharp}_{\varphi_{\pi}}}}\cdot \frac{\abs{\Gamma_{\mathfrak{s}^{\vee},\pi_J}}}{\dim(\sigma)}.$$
\end{lemma}

\begin{proof}

   The $L$-parameter $\varphi_\pi$ is discrete, so $Z_{(G/Z(G)_s)^{\vee}}(\varphi_\pi)$ is finite, therefore $S^{\sharp}_{\varphi_\pi}=Z_{(G/Z(G)_s)^{\vee}}(\varphi_\pi)$. Again from the discussion in Section \ref{ReductionSection}, we have 
   
    $$\rho_\pi=\Ind_{\pi_0(Z_{J}(\varphi_{\pi_J}))_{\rho_{\pi_J}}}^{\pi_0(Z_{J}(\varphi_{\pi_J}))}(\rho_{\pi_J}\otimes \sigma),$$ 
    where $\pi_0(Z_{J}(\varphi_{\pi_J}))_{\rho_{\pi_J}}$ denotes the stabilizer of $\rho_{\pi_J}$ in $\pi_0(Z_{J}(\varphi_{\pi_J}))$. Then 
    \begin{align*}
    \dim(\rho_\pi)=&\dim(\rho_{\pi_J})\cdot\dim(\sigma)\cdot[\pi_0(Z_{J}(\varphi_{\pi_J})):\pi_0(Z_{J}(\varphi_{\pi_J}))_{\rho_{\pi_J}}]\\ =&\dim(\rho_{\pi_J})\cdot\dim(\sigma)\cdot\frac{\abs{\pi_0(Z_{J}(\varphi_{\pi_J}))}}{\abs{\pi_0(Z_{J}(\varphi_{\pi_J}))_{\rho_{\pi_J}}}}
    \end{align*}

    But from \cite[Lemma 3.12(b)]{GradDisc}
    $$\frac{\abs{\pi_0(Z_J(\varphi_{\pi_J}))_{\rho_{\pi_J}}}}{\abs{\pi_0(Z_{J^{\circ}}(\varphi_{\pi_J}))}}=\abs{\Stab_{\Gamma_{\mathfrak{s}^{\vee}}}(\pi_J)}$$  

    So, what is left to prove is: 
    \begin{equation}\label{LastEquation}\abs{S^{\sharp}_{\varphi_{\pi_J}}}=\abs{S_{\varphi_\pi}}\frac{\abs{\pi_0(Z_{J^{\circ}}(\varphi_{\pi_J}))}}{\abs{\pi_0(Z_{J}(\varphi_{\pi_J}))}}.\end{equation}

    In fact, if we prove it, we conclude
    $$\frac{\dim(\rho_{\pi_J})}{\abs{S^{\sharp}_{\varphi_{\pi_J}}}}=\frac{\dim(\rho_{\pi})}{\abs{S^{\sharp}_{\varphi_{\pi_J}}}}\cdot \frac{1}{\dim(\sigma)}\cdot \frac{\abs{\pi_0(Z_J(\varphi_{\pi_J}))_{\rho_{\pi_J}}}}{\abs{\pi_0(Z_{J^{\circ}}(\varphi_{\pi_J}))}} = \frac{\dim(\rho_{\pi})}{\abs{S^{\sharp}_{\varphi_{\pi}}}}\cdot \frac{\abs{\Gamma_{\mathfrak{s}^{\vee},\pi_J}}}{\dim(\sigma)} .$$

We have 
$$Z_{G^{\vee}}(\varphi_\pi))=Z_{Z(G)_sG^{\vee}_{\der}Z(G^{\vee})_{\an}}(\varphi_\pi)=Z(G)_sZ_{G^{\vee}_{\der}Z(G^{\vee})_{\an}}(\varphi_\pi)$$ therefore, the difference between $S^{\sharp}_{\varphi_\pi}$ and $\pi_0(Z_{G^{\vee}}(\varphi_\pi))$ consists only of elements in $Z(G)_s\cap G_{\der}^{\vee}Z(G^{\vee})_{\an}.$ Since $\varphi_\pi$ is discrete,  $Z_{G^{\vee}_{\der}Z(G^{\vee})_{\an}}(\varphi_\pi)$ is finite and 
$$\pi_0(Z_{G^{\vee}}(\varphi_\pi))=Z_{G^{\vee}_{\der}Z(G^{\vee})_{\an}}(\varphi_\pi)/(Z(G^{\vee})_s\cap G_{\der}^{\vee}Z(G^{\vee})_{\an}).$$ A similar argument can be done for $J^{\circ},$ so in order to prove \eqref{LastEquation}, we only need to prove 
\begin{equation}\label{verylastequation}
   Z(G^{\vee})_s\cap G_{\der}^{\vee}Z(G^{\vee})_{\an}=Z(J^{\circ})_s\cap J_{\der}^{\circ}Z(J^{\circ})_{\an}.
\end{equation}

Notice that $Z(G^{\vee})_s=Z(J^{\circ})_s$ since they have the same dimension. First we prove that the right hand side is contained in the left hand side. Let $x\in Z(J^{\circ})_s\cap G_{\der}^{\vee}Z(G^{\vee})_{\an}$ and write it as $x=yz\in G_{\der}^{\vee}Z(G^{\vee})_{\an}.$ Then $y= xz^{-1}\in J$ since $Z(G^{\vee})_{\an}\subset Z(J^{\circ})_{\an},$ and from $$J^{\circ}_\der=J^{\circ}\cap G_\der^{\vee} $$  we have $y\in J^{\circ}_\der.$ 

On the other hand, $J^{\circ}_{\der}Z(J^{\circ})_{\an}\subset G_{\der}^{\vee}Z(G^{\vee})_{\an}$. In fact, $J_\der^{\circ}$ is obviously in $G_\der^{\circ},$ and $\Lie(Z(J^{\circ})_{\an})\subset\Lie(G_\der^{\vee})\oplus \Lie(Z(G^{\vee})_{\an}).$ To check this we notice that at the level of Lie algebras we have $$\Lie(J^{\circ})=\Lie(J^{\circ}_{\der})\oplus \Lie(Z(J^{\circ})_{\an})\oplus \Lie(Z(J^{\circ})_{s})\subset \Lie(G^{\vee}),$$ but $$\Lie(G^{\vee})= \Lie(G^{\vee}_{\der})\oplus \Lie(Z(G^{\vee})_{\an})\oplus \Lie(Z(G^{\vee})_{s}),$$ 

so $$\Lie(Z(J^{\circ})_{\an})\subset\Lie(G_\der^{\vee})\oplus \Lie(Z(G^{\vee})_{\an}).$$

Thus, \eqref{verylastequation} holds, which implies \eqref{LastEquation} and completes the proof of the lemma.\qedhere

\end{proof}

\begin{theorem}
    Let $\pi$ be a discrete series representation contained in $\Rep(G)^{\mathfrak{s}},$ and consider the local Langlands correspondence for principal series representations of quasi-split groups from \cite{quasisplitps}. Then, Conjecture \ref{HII} holds, meaning that we have 
    $$\fdeg(\pi)=\frac{\dim(\rho_\pi)}{\abs{S^{\sharp}_{\varphi_\pi}}}\abs{\gamma(0,\Ad_{G^{\vee}}\circ\varphi_{\pi},\psi)}.$$
\end{theorem}

\begin{proof}
     The proof relies on the fact that the formal degree conjecture is already known for $\pi_{J}$, since it corresponds to a unipotent representation of $\mathcal{J}(F)$ \cite{UnipotentFormalDegree}. Therefore, the formal degree of the representation of $\mathcal{J}(F)$ corresponding to $\pi_J$ is equal to 
    \begin{equation}\label{UnipFD}
    \fdeg(\pi_J)=\frac{\dim(\rho_{\pi_J})}{\abs{S^{\sharp}_{\varphi_{\pi_J}}}}\abs{\gamma(0,\Ad_{J^{\circ}}\circ\varphi_{\pi_J},\psi)}.
    \end{equation}
    
    The difference between the formal degree of $\pi_J$ and the formal degree of the associated $\mathcal{J}(F)$-representation consists of just the volume of a type, which in this case is just the volume of the Iwahori subgroup $I_{\mathcal{J}}.$ 

    So far, using \eqref{firstFD},\eqref{UnipFD} and Lemma \ref{Lemma1} we have: 
    \begin{align*}
    \fdeg(\pi) &= \frac{\fdeg(\pi_J)\cdot\dim(\sigma)}{\vol(J_\chi)\cdot \abs{\Gamma_{\mathfrak{s}^{\vee},\pi_J}}} \\ &
    = \frac{\dim(\sigma)}{\abs{\Gamma_{\mathfrak{s}^{\vee},\pi_J}}}\frac{\vol(I_{\mathcal{J}})}{\vol(J_{\chi})} \frac{\dim(\rho_{\pi_J})}{|S_{\varphi_{\pi_J}}^{\sharp}|} \cdot \left| \gamma(0, \Ad_{J^{\circ}} \circ \varphi_{\pi_J},\psi) \right|\\ 
    &=
    \frac{\vol(I_{\mathcal{J}})}{\vol(J_{\chi})} \frac{\dim(\rho_{\pi})}{\abs{S_{\varphi_{\pi}}^{\sharp}}} \cdot \left| \gamma(0, \Ad_{J^{\circ}} \circ \varphi_{\pi_J},\psi) \right|.
    \end{align*}

    Since the absolute value of the $\gamma$-factor is not affected by the choice of the nilpotent operator \cite[Proposition A.1]{UnipotentFormalDegree}, we have
    $$\abs{\gamma(0,\Ad_{G^{\vee}}\circ\varphi_\pi,\psi)}=\abs{\gamma(0,\Ad_{G^{\vee}}\circ\varphi_\chi,\psi)},$$
    and $$\abs{\gamma(0,\Ad_{J^{\circ}}\circ\varphi_{\pi_J},\psi)}=\abs{\gamma(0,\Ad_{J^{\circ}}\circ\widetilde{\varphi_\chi},\psi)}.$$
    Notice that by construction $\Lie(J^{\circ})=(\mathfrak{g}^{\vee})^{\varphi_\chi(I_F)},$ and we have an equality $$L(s,\Ad_{G^{\vee}}\circ\varphi_\chi)=L(s,\Ad_{J^{\circ}}\circ\widetilde{\varphi_{\chi}}).$$ 
    Therefore,
    $$\abs{\gamma(0,\Ad_{G^{\vee}}\circ\varphi_\chi,\psi)}=\abs{\gamma(0,\Ad_{J^{\circ}}\circ\widetilde{\varphi_\chi},\psi)}\cdot\abs{\varepsilon(0,\Ad_{G^{\vee}}\circ\varphi_\chi,\psi)},$$
    and using Lemma \ref{Lemma2}, we get 
    \begin{equation*}
    \fdeg(\pi)=\frac{\dim(\rho_\pi)}{\abs{S^\sharp_{\varphi_\pi}}}\cdot\left|\gamma(0,\Ad_{G^{\vee}}\circ\varphi_\pi,\psi)\right|.
    \qedhere
\end{equation*}
\end{proof}